\documentclass{article}
\usepackage[utf8]{inputenc}
\usepackage{amsmath,amsthm,amsfonts,amssymb,epsfig,titlesec,float,caption,mathrsfs,enumitem}
\usepackage{aliascnt}
\usepackage{hyperref}
\usepackage{tikz-cd}
\usepackage{tikz}

\usetikzlibrary{patterns, shadings}
\usepackage{xcolor}
\usepackage{verbatim}
\usepackage{accents}
\usepackage[citestyle=alphabetic,bibstyle=alphabetic,backend=bibtex]{biblatex}
\usepackage[disable]{todonotes}
\usepackage[american]{babel}

\usepackage{wrapfig}
\restylefloat{figure}
\usepackage[left=1in,top=1in,right=1in]{geometry}

\usepackage[capitalize]{cleveref}
\usepackage{subcaption}

\newtheorem*{theorem*}{Theorem}
\crefname{thm}{Theorem}{Theorems}
\Crefname{thm}{Theorem}{Theorems}
\newaliascnt{lemma}{thm}
\newtheorem{lemma}[lemma]{Lemma}
\aliascntresetthe{lemma}
\crefname{lemma}{Lemma}{Lemmas}
\Crefname{lemma}{Lemma}{Lemmas}
\newaliascnt{prop}{thm}
\newtheorem{prop}[prop]{Proposition}
\aliascntresetthe{prop}
\crefname{prop}{Proposition}{Propositions}
\Crefname{prop}{Proposition}{Propositions}
\newaliascnt{corollary}{thm}
\newtheorem{corollary}[corollary]{Corollary}
\aliascntresetthe{corollary}
\crefname{corollary}{Corollary}{Corollaries}
\Crefname{corollary}{Corollary}{Corollaries}
\newaliascnt{conjecture}{thm}
\newtheorem{conjecture}[conjecture]{Conjecture}
\aliascntresetthe{conjecture}
\crefname{conjecture}{Conjecture}{Conjectures}
\Crefname{conjecture}{Conjecture}{Conjectures}
\newaliascnt{notation}{thm}

\aliascntresetthe{notation}
\crefname{notation}{Notation}{Notations}
\Crefname{notation}{Notation}{Notations}
\newaliascnt{question}{thm}

\aliascntresetthe{question}
\crefname{question}{Question}{Questions}
\Crefname{question}{Question}{Questions}
\newaliascnt{obs}{thm}

\aliascntresetthe{obs}
\crefname{obs}{Observation}{Observations}
\Crefname{obs}{Observation}{Observations}
\newaliascnt{assumption}{thm}

\aliascntresetthe{assumption}
\crefname{assumption}{Assumption}{Assumptions}
\Crefname{assumption}{Assumption}{Assumptions}
\newaliascnt{setting}{thm}

\aliascntresetthe{setting}
\crefname{setting}{Setting}{Settings}
\Crefname{setting}{Setting}{Settings}
\newaliascnt{construction}{thm}
\newtheorem{construction}[construction]{Construction}
\aliascntresetthe{construction}
\crefname{construction}{Construction}{Constructions}
\Crefname{construction}{Construction}{Constructions}
\newaliascnt{outline}{thm}

\aliascntresetthe{outline}
\crefname{outline}{Outline}{Outlines}
\Crefname{outline}{Outline}{Outlines}
\newaliascnt{philosophy}{thm}

\aliascntresetthe{philosophy}
\crefname{philosophy}{Philosophy}{Philosophies}
\Crefname{philosophy}{Philosophy}{Philosophies}

\theoremstyle{remark} 
\newaliascnt{remark}{thm}

\aliascntresetthe{remark}
\crefname{remark}{Remark}{Remarks}
\Crefname{remark}{Remark}{Remarks}
\newaliascnt{example}{thm}
\newtheorem{example}[example]{Example}
\aliascntresetthe{example}
\crefname{example}{Example}{Examples}
\Crefname{example}{Example}{Examples}

\theoremstyle{definition} 
\newaliascnt{definition}{thm}
\newtheorem{definition}[definition]{Definition} 
\aliascntresetthe{definition}
\crefname{definition}{Definition}{Definitions}
\Crefname{definition}{Definition}{Definitions}

\titleformat*{\section}{\normalsize \bfseries \filcenter}
\titleformat*{\subsection}{\normalsize \bfseries }
\crefname{mainthm}{Theorem}{Theorems}
\Crefname{mainthm}{Theorem}{Theorems}

\crefname{maincor}{Corollary}{Corollaries}
\Crefname{maincor}{Corollary}{Corollaries}

\makeatletter
\def\namedlabel#1#2{\begingroup
   \def\@currentlabel{#2}\label{#1}\endgroup
}
\makeatother
 \newcommand{\eps}{\varepsilon}
\newcommand{\RR}{\mathbb R}
\newcommand{\ZZ}{\mathbb Z}
\newcommand{\CC}{\mathbb C}

\newcommand{\NN}{\mathbb N}

\renewcommand{\Re}{\text{Re}}

\DeclareMathOperator{\Area}{Area}

\DeclareMathOperator{\Flux}{Flux}

\def\ul{\underline}

\title{Some cute applications of Lagrangian cobordisms towards examples in quantitative symplectic geometry}
\author{Jeff Hicks and Cheuk Yu Mak }
\date{\today}

\begin{document}
\maketitle

\begin{abstract}
    We provide some constructions using Lagrangian cobordisms that improve known examples for some symplectic squeezing problems. Additionally, we prove a flexibility result that Lagrangian submanifolds that are Lagrangian isotopic are also Lagrangian cobordant.
 \end{abstract}
\section{Introduction}
     A general type of problem in symplectic geometry is the ``squeezing problem''. A typical example of a squeezing problem asks: given subsets $A, B$ inside of a symplectic manifold $(X, \omega)$, does there exist a symplectic or Hamiltonian isotopy $\phi: X\to X$ so that $\phi(A)\subset B$.
 In the simplest cases one can use quantitative information from the symplectic form to solve a squeezing problem. For example, a first quantity that one can associate with a squeezing problem comes from the volume, which must satisfy the inequality $\int_A\omega^n\leq \int_B \omega^n$. It is natural to ask whether we can find a solution to a squeezing problem given some quantitative information about the objects being squeezed. These questions have two components:  constructing examples of squeezings and finding obstructions to the existence of squeezings.
 
 In this paper, we give some improvements on known constructions of squeezings by using Lagrangian cobordisms to produce examples. In each problem, the quantitative nature of the construction exhibits itself by considering the shadow of the Lagrangian cobordisms used.

\subsubsection*{Problem I: Packing Lagrangian Tori} 

\begin{wrapfigure}{r}{.4\textwidth}
    \begin{center}
        \begin{tikzpicture}

\fill[gray!20]  (0,0) rectangle (4.5,2.5) node (v1) {};
\node[red] at (1,1) {$\times$};
\node[red] at (1,2) {$\times$};
\node[red] at (2,1) {$\times$};
\node[red] at (2,2) {$\times$};
\node[red] at (3,1) {$\times$};
\node[red] at (3,2) {$\times$};
\node[red] at (4,1) {$\times$};
\node[red] at (4,2) {$\times$};
\draw (0,2.5) -- (0,0) -- (4.5,0);
\draw[dashed] (0,2.5) -- (4.5, 2.5) -- (4.5,0);
\draw[|-|] (-0.5,2.5) -- node[fill=white]{$b$} (-0.5,0);
\draw[|-|] (0,-0.5) -- node[fill=white]{$a$} (4.5,-0.5);
\end{tikzpicture}         \caption{Can you pack an integral Lagrangian torus into the polydisk so that it avoids all other integral Lagrangian tori?}
    \end{center}
\end{wrapfigure}
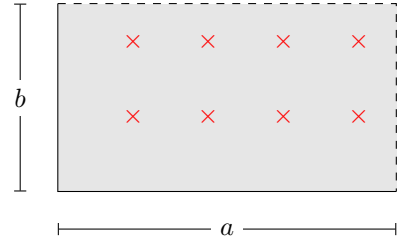

The integral Lagrangian packing problem asks how many pairwise disjoint integral Lagrangians\footnote{An integral Lagrangian is a Lagrangian submanifold such that the area homomorphism $H_2(X,L;\mathbb{Z}) \to \mathbb{R}$ takes only integer values.} 
one can find in a symplectic manifold $(X,\omega)$.

This problem has been studied by \cite{hind2021packing} for $4$-dimensional polydisks. For $a, b>0$, let the Lagrangian product torus be 
\[L(a, b):=\{(z_1, z_2)\in \CC^2 \;|\; \pi |z_1|^2=a, \pi |z_2|^2=b\}.\]
Define the polydisk to be 
\[P(a, b):=\{(z_1, z_2)\in \CC^2 \;|\; \pi |z_1|^2<a, \pi |z_2|^2<b\}.\]
Fix $a, b>0$. For any $m, n\in \mathbb{N}$ with $m<a$ and $n<b$, $L(m,n)$ is an integral Lagrangian in $P(a,b)$. 
If every integral Lagrangian $L \subset P(a,b)$ intersects some $L(m,n)$, we say that the integral packing of $P(a, b)$ by $\{L(m,n)\}$ is maximal.
    The Lagrangian torus packing problem has the following obstruction and construction:
    \begin{description}
    \item[Obstruction {\cite[Theorem 1.1]{hind2021packing}}:] If $1<a, b<2$, the integral packing of $P(a, b)$ by $L(1,1)$ is maximal. 
    
    \item[Construction {\cite[Theorem 1.2]{hind2021packing}}:] If $\min(a, b)>2$, then the integral packing of $P(a, b)$ by $\{L(m,n)\}$ is not maximal.
\end{description}
In \cref{sec:packing} we provide the following construction.
\begin{construction}\label{t:1}
    If $2<a$ and $1<b$, then the integral packing of $P(a, b)$ by $\{L(m,n)\}$ is not maximal.
\end{construction}

As a consequence, we answer \cite[Question 1.5]{hind2021packing} negatively and \cite[Question 1.6 and 1.7]{hind2021packing} affirmatively (see Corollary \ref{c:packings}).

\subsubsection*{Problem II: Displacing Lagrangian Tori} 
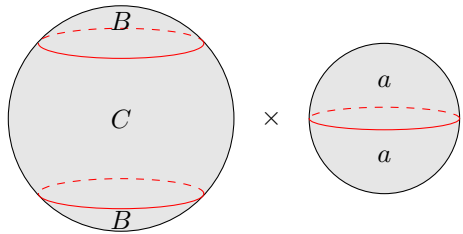
\begin{wrapfigure}{r}{.4\textwidth}
    \begin{center}
        \begin{tikzpicture}

\draw[fill=gray!20]  (1,0.5) ellipse (1.5 and 1.5);
\draw[fill=gray!20]  (4.5,0.5) ellipse (1 and 1);

\begin{scope}[]

\begin{scope}[]
\clip  (2.5,2) rectangle (-0.5,1.5);
\draw[red, dashed]  (1,1.5) ellipse (1.1 and 0.2);
\end{scope}

\begin{scope}[]
\clip  (2.5,1.5) rectangle (-0.5,1);
\draw[red]  (1,1.5) ellipse (1.1 and 0.2);
\end{scope}
\end{scope}

\begin{scope}[shift={(3.5,-1)}]

\begin{scope}[]
\clip  (2.5,2) rectangle (-0.5,1.5);
\draw[red, dashed]  (1,1.5) ellipse (1 and 0.15);
\end{scope}

\begin{scope}[]
\clip  (2.5,1.5) rectangle (-0.5,1);
\draw[red]  (1,1.5) ellipse (1 and 0.15);
\end{scope}
\end{scope}

\begin{scope}[shift={(0,-2)}]

\begin{scope}[]
\clip  (2.5,2) rectangle (-0.5,1.5);
\draw[red, dashed]  (1,1.5) ellipse (1.1 and 0.2);
\end{scope}

\begin{scope}[]
\clip  (2.5,1.5) rectangle (-0.5,1);
\draw[red]  (1,1.5) ellipse (1.1 and 0.2);
\end{scope}
\end{scope}
\node at (3,0.5) {$\times$};
\node at (1,-0.85) {$B$};
\node at (1,0.5) {$C$};
\node at (1,1.8) {$B$};
\node at (4.5,1) {$a$};
\node at (4.5,0) {$a$};
\end{tikzpicture}         \caption{Can you displace the red Lagrangian, a disjoint union of two tori, from itself?}
    \end{center}
\end{wrapfigure}
The second problem we consider is displacing a Lagrangian link from itself in $S^2\times S^2$, studied by the second author and Smith in \cite{mak2021non}.

Let $\omega$ be a symplectic form on $S^2$ with area $1$.
For any $A \in \mathbb{R}_{>0}$, we denote $(S^2, A \omega)$ by $S^2_A$.
Let $L_1$ and $L_2$ be disjoint embedded circles in $S^2_{2B+C}$ such that the complement of 
$\underline{L}_{B, C}:=L_1 \cup L_2$ consists of two disks of area $B$ and one annulus of area $C$.
Let $K_a$ be the equator in $S^2_{2a}$.

It is clear by area considerations that every connected component of the Lagrangian $\underline{L}_{B, C}$ in $S^2_{2B+C}$ is Hamiltonian non-displaceable from $\underline{L}_{B, C}$ if and only if $B-C \ge 0$.
The following result concerns $\underline{L}_{B, C} \times K_a$.
\begin{description}

\item[Obstruction {\cite[Theorem 1.1]{mak2021non}}:] If $B-C \ge a >0$, any connected component of the Lagrangian $\underline{L}_{B, C} \times K_a$ in $S^2_{2B+C} \times S^2_{2a}$ is Hamiltonian non-displaceable from $\underline{L}_{B, C} \times K_a$.

\item[Construction:] In contrast, $\underline{L}_{B, C} \times K_a$ is Hamiltonian displaceable from itself when $a>B$.
\end{description}

In fact, the obstruction works for collections of more parallel circles rather than $\underline{L}_{B, C}$ (see \cite{polterovich2021lagrangian}). On the construction side, Polterovich showed that $\underline{L}_{B, C} \times K_a$ is Hamiltonian displaceable when $a$ is sufficiently large \cite[Example 6.3.C]{PolBook}, \cite[Lemma 1.11]{mak2021non}. Dimitroglou Rizell explained to us the construction above which only requires $a>B$, and is a direct application of probes \cite[Lemma 2.4]{McDprobe} (see Figure \ref{fig:probe}).

While the obstruction has a dependence on $C$, the construction does not. We provide a construction that depends on $C$.
\begin{construction}\label{t:2}
    $\underline{L}_{B, C} \times K_a$ is displaceable from one of its connected components when $a>2(B-C)$.
\end{construction}

The particularly interesting case is when $B=C$. In this case, $\underline{L}_{B, C}$ is a monotone link in the sense of \cite{CGHMSS21}. Recall that a union of pairwise disjoint circles $\ul{L}=L_1 \cup \dots \cup L_k$ in a closed symplectic surface $\Sigma$ is called a {\it monotone link} if the connected components of the complement of $\underline{L}$ have the same area. Every monotone link is non-displaceable by area consideration, and it has a non-zero Heegaard Floer type invariant which plays an important role in \cite{CGHMSS21}.

As a generalization of Construction \ref{t:2}, we show that
\begin{construction}\label{t:2'}
    Let $\underline{L}$ be a monotone link in a closed symplectic surface $\Sigma$.
    If it has more than one component and there is a component $L_1 \subset \underline{L}$ which bounds a disk that is disjoint from the other components, then $L_1 \times K_a$ is displaceable from $\underline{L} \times K_a$ in $\Sigma \times S^2_{2a}$ for any $a>0$.
\end{construction}

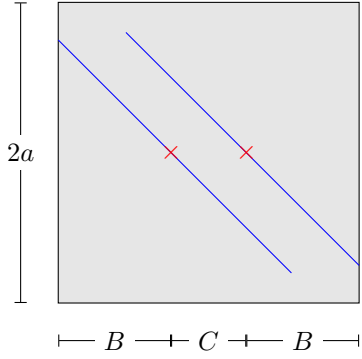
\begin{figure}[ht]
    \centering
    \begin{tikzpicture}

\draw[|-|] (-1.5,2.5) -- node[fill=white]{$2a$} (-1.5,-1.5);

\draw[|-|] (3,-2) -- node[fill=white]{$B$} (1.5,-2);

\draw[|-|] (1.5,-2) -- node[fill=white]{$C$} (0.5,-2);

\draw[|-|] (0.5,-2) -- node[fill=white]{$B$} (-1,-2);

\draw[fill=gray!20]  (-1,2.5) rectangle (3,-1.5);

\draw[blue] (-1,2) -- (2.1,-1.1);
\draw[blue] (3,-1) -- (-0.1,2.1);

\node[red] at (0.5,0.5) {$\times$};
\node[red] at (1.5,0.5) {$\times$};
\end{tikzpicture}     \caption{Toric picture of $S^2_{2B+C} \times S^2_{2a}$. The blue line segments are probes that displace $\underline{L}_{B, C} \times K_a$ (red crosses).}
    \label{fig:probe}
\end{figure}

\subsubsection*{Problem III: Constructing Lagrangian Submanifolds} 

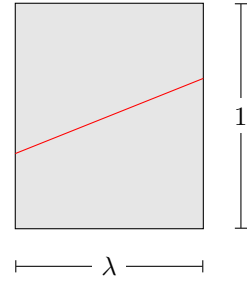
\begin{wrapfigure}{r}{.4\textwidth}
    \begin{center}
        \begin{tikzpicture}[rotate=90]

\draw[fill=gray!20]  (0.5,-1) rectangle (3.5,-3.5);
\draw[red] (2.5,-3.5) -- (1.5,-1);
\draw[|-|] (0,-1) -- node[fill=white]{$\lambda$} (0,-3.5);
\draw[|-|] (0.5,-4) -- node[fill=white]{1} (3.5,-4);
\end{tikzpicture}         \caption{Can you squeeze the Lagrangian Klein bottle in $S^2_\lambda \times S^2_1$ as you increase $\lambda$ beyond 2?}
        \label{fig:lagrangianKleinBottle}
    \end{center}
\end{wrapfigure}

The third problem we examine constructs non-orientable Lagrangian submanifolds in $S^2\times S^2$.

Let $X$ be a 4-manifold with symplectic form $\omega$, and $\beta\in H_2(X, \ZZ/2\ZZ)$ be a homology class. The minimal genus problem asks ``what is the minimal (non-orientable) genus of a (possibly non-orientable) Lagrangian submanifold representing the class $\beta$?'' We denote this quantity $\eta(X, \omega, \beta)$.

In \cite{evans2021lagrangian}, Evans explores the behavior of $\eta(X, \omega, \beta)$ as $\omega$ is varied with fixed $X, \beta$. In the specific setting where $X=S^2_\lambda\times S^2_1$ with symplectic form  $\omega_\lambda=\lambda\omega_{S^2}+\omega_{S^2}$ and $\beta=[S^2\times \{\bullet\}]$, he conjectures that 
\begin{equation}
\lim_{\lambda\to\infty} \eta(X, \omega_\lambda, \beta)=\infty.
\label{conj:evans}
\end{equation}
As evidence toward the conjecture, Evans provides the following obstruction and construction. 
\begin{description}
    \item[Obstruction {\cite[Theorem 3.1]{evans2021lagrangian}}:] Consider the symplectic space $S^2_\lambda\times B^*_1S^1$, the product of a sphere with area $\lambda$ and a cylinder of area $1$. For any $\epsilon>0$, there does not exist a family of embedded Lagrangian Klein bottles $i_{\lambda}: L\to S^2_\lambda\times B^*_1S^1$ where $\lambda \in [2-\eps, 2+\eps]$, and $i_{2-\epsilon}(L)$ is the visible Lagrangian Klein bottle drawn in Figure \ref{fig:lagrangianKleinBottle}.
    \item[Construction {\cite[Lemma 2.9]{evans2021lagrangian}}:] The minimal genus is bounded by 
    \[\eta(X, \omega_\lambda, \beta) \le 20\ell+2 \text{ when $\lambda<10\ell+2$}.\]
\end{description}
It was recently shown in \cite{adaloglou2024lagrangiankleinbottless2} that there is no Lagrangian Klein bottle in the class of $\beta$ if $\lambda>2$.

The construction of \cite[Lemma 2.9]{evans2021lagrangian} uses Lagrangian lifts of tropical curves; the jumps in the bound correspond to when the rectangle with side lengths $(1, \lambda)$ can squeeze in another tropical curve. 
Using Lagrangian cobordisms, we provide a (slightly) improved bound whose jumps occur at even integer values of $\lambda$. 
\begin{construction}\label{t:3}
    $\eta(X, \omega_\lambda, \beta) \le 4n+2$, where $\lambda<2n$.
\end{construction}
By a similar construction, we obtain the following lemma, which might be of independent interest.
\begin{lemma}[=Corollary \ref{c:LagIsotopic}]
    Suppose that $L_0, L_1$ are orientable and Lagrangian isotopic. Then $L_0$ and $L_1$ are Lagrangian cobordant.
\end{lemma}

We will only apply this lemma to the case that $L_0$ and $L_1$ are unions of circles. The non-orientable Lagrangians we construct will be a concatenation of Lagrangian cobordisms.

\subsection*{Acknowledgements}
We thank Richard Hind for his interest in this work, and Felix Schlenk and Egor Shelukhin for some helpful comments. We would also like to thank the organizers of the Uppsala University-Nantes Workshop on Lagrangian cobordisms and Floer theory for providing a stimulating environment where we first discussed these constructions. Finally, we thank an anonymous referee for their helpful and detailed comments.
JH was supported by ERC Grant 850713 and EPSRC Grant EP/Z53528X/1. CM was partially supported by the Simons Collaboration on Homological Mirror Symmetry and a Royal Society University Research Fellowship. 
 \subsection{Background: Lagrangian Cobordisms}
    
Quantitative symplectic geometry looks at how the answers to the above problems change as we vary the quantities in the problem. 
In general, these problems are studied by producing constructions (providing an affirmative answer for some subset of the quantitative data) or developing obstructions (which can answer the question in the negative for some subsets of the quantitative data). 
In this paper, we will use Lagrangian cobordisms to provide constructions for several examples of problems coming from quantitative symplectic geometry. 
A feature of these constructions is that the shadow metric of the Lagrangian cobordism provides a connection to the quantitative nature of the problem.
\begin{definition}[\cite{arnol1980lagrange,biran2013lagrangian}]
    Let $L_0, \ldots, L_k$ be Lagrangian submanifolds of $X$. 
    A \emph{Lagrangian cobordism} from $\{L_j\}_{j=1}^k$ to $L_0$ is a Lagrangian submanifold $K\subset X\times \CC$ which satisfies the following conditions:
    \begin{itemize}
        \item \emph{Fibered over ends:} There exist constants $t^-<t^+$ such that 
        \[
            K\cap \{(x, z)\;:\; \Re(z)\geq t^+\}= \bigcup_{j=1}^k L_j \times \{ j\cdot \sqrt{-1} + \RR_{\ge t^+}\}
        \]
        \[
            K\cap \{(x, z)\;:\; \Re(z)\leq t^-\}= L_0 \times \RR_{\leq t^-}
        \]
        \item \emph{Compactness:} There exists a constant $s>0$ so that the projection $\pi_{\sqrt{-1}\RR}: K\to \sqrt{-1} \RR \subset \CC$ is contained within an interval $[-\sqrt{-1} s, \sqrt{-1} s]$. 
    \end{itemize}
\end{definition}
We will write such a cobordism as $K:(L_1, \ldots, L_k)\rightsquigarrow L_0$.
The regular level sets of the real coordinate $t=\Re(z)$ in $K$ project to immersed Lagrangian submanifolds of $X$.
\begin{definition}
    \label{claim:slice}
Given a Lagrangian cobordism $K\subset X\times \CC$, and $t\in \RR$ a value so that $\pi_\RR^{-1}(t)\subset X\times \CC$ is transverse to $K$, the \emph{slice} of $K$ at $t\in\RR$ is an immersed Lagrangian submanifold 
\[K|_t:=\pi_X(\pi_\RR^{-1}(t)\cap K ).\]
\end{definition}

Given a Lagrangian cobordism $K\subset X\times \CC$, the \emph{shadow} $\Area(K)$ is the infimum of the area of simply connected regions $U\subset \CC$ that contain $\pi_\CC(K)$ (\cite{cornea2019lagrangian}).

\begin{example}\label{e:suspensionarea}
Let $H=(H_t)_{t \in \RR}: X\times \RR\to \RR$ be a time-dependent Hamiltonian with compact support, and $i:L\to X$ be a Lagrangian submanifold. Let $\phi_t: X\to X$ be the time $t$ Hamiltonian flow, and write $i_t:= \phi_t\circ i$. The \emph{suspension} of $H$ is the Lagrangian cobordism $K_{H}$ which is parameterized by 
\begin{align}
    K_{H}:= L\times \RR\to& X \times \CC \label{eq:suspension}\\
    (q, t) \mapsto& (\phi_t(i(q)),t+\sqrt{-1} i_t^*H_t(q))
\end{align}
It satisfies the property that $\Area(K_{H})=\|H_t\|_L$, where $\|H_t\|_L$ is the (relative) \emph{Hofer norm}
\[\|H_t\|_L:= \int_{\RR} \left(\sup\{i_t^*H_t(q)\;|\; q\in L\} - \inf \{ i_t^*H_t(q)\;|\; q\in L\}\right)\, dt\]
as studied by \cite{hofer1990topological,mcduff2001geometric,PolBook,chekanov1996hofer}.
\end{example}
We say that a Lagrangian homotopy $i_t: L\times \RR\to X$ is an \emph{exact homotopy} if the flux class is exact, with primitive $H_t: L\times \RR\to \RR$. In this setting, \cref{eq:suspension} also yields a Lagrangian cobordism between $i(t\ll 0)$ and $i(t\gg 0)$.
\begin{example}
    Let $L_1, L_2$ be Lagrangian submanifolds intersecting transversely at a single point $q$. The \emph{surgery trace cobordism} is a Lagrangian cobordism
    \[K: (L_1, L_2) \rightsquigarrow L_1\#_q L_2,\] where $L_1\#_q L_2$ is the Polterovich surgery (\cite{polterovich1991surgery,lalonde1991sous}) of $L_1$ and $L_2$ at $q$ (see \cref{fig:lagrangiansurgerytrace}). The shadow of $K$ is equal to the neck width of the Polterovich surgery.
    Note that the order of the summands in the Polterovich surgery is determined by the ordering of the ends of the surgery trace cobordism.
    \label{exam:surgery}
\end{example}
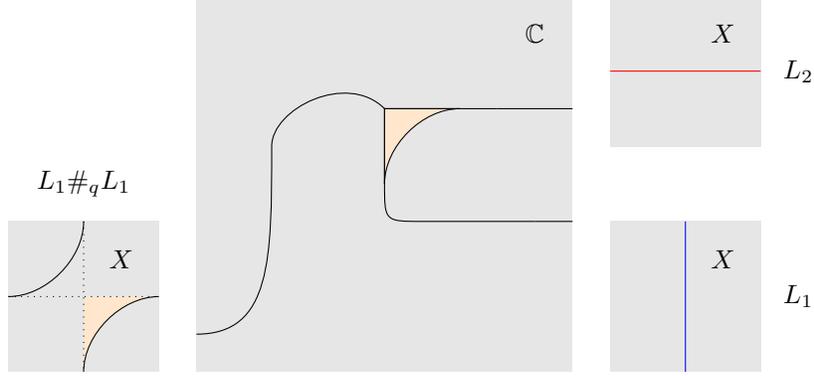
\begin{figure}
    \centering
    \begin{tikzpicture}

    \fill[gray!20]  (-1,-0.5) rectangle (4,-5.5);
    \begin{scope}[shift={(-1,-3.5)}]
    \fill[gray!20]  (-2.5,0) rectangle (-0.5,-2);
    \fill[orange!20] (-1.5,-2) .. controls (-1.5,-1.5) and (-1.5,-1) .. (-1.5,-1) .. controls (-1.5,-1) and (-1,-1) .. (-0.5,-1) .. controls (-1,-1) and (-1.5,-1.5) .. (-1.5,-2);
    
    \draw (-1.5,-2) .. controls (-1.5,-1.5) and (-1,-1) .. (-0.5,-1);
    \draw (-1.5,0) .. controls (-1.5,-0.5) and (-2,-1) .. (-2.5,-1);
    
    \draw[dotted] (-2.5,-1) -- (-0.5,-1) (-1.5,0) -- (-1.5,-2);
    \end{scope}
    
    \fill[gray!20]  (4.5,-3.5) rectangle (6.5,-5.5);
    \fill[gray!20]  (4.5,-0.5) rectangle (6.5,-2.5);
    \draw (1.5,-3) -- (1.5,-2);
    \draw (3,-2) -- (1.5,-2);
    \draw[fill=orange!20] (1.5,-3) .. controls (1.5,-2.5) and (2,-2) .. (2.5,-2) .. controls (2,-2) and (1.5,-2) .. (1.5,-2) .. controls (1.5,-2) and (1.5,-2) .. (1.5,-3);
    \draw (1.5,-2) -- (1.5,-2);
    \draw[blue] (5.5,-3.5) -- (5.5,-5.5);
    \draw[red](4.5,-1.5) -- (6.5,-1.5);

    \draw (1.5,-3) .. controls (1.5,-3.5) and (1.5,-3.5) .. (2,-3.5) .. controls (2.5,-3.5) and (2.5,-3.5) .. (3.5,-3.5);
    \draw (3,-2) -- (4,-2);
    \draw (3.5,-3.5) -- (4,-3.5);
    \draw (1.5,-2) .. controls (1,-1.5) and (0,-2) .. (0,-2.5) .. controls (0,-4) and (0,-5) .. (-1,-5);
    \node at (7,-4.5) {$L_1$};
    \node at (7,-1.5) {$L_2$};
    \node at (-2.5,-3) {$L_1\#_q L_2$};
    \draw (-3.5,-4);
    \node at (3.5,-1) {$\mathbb C$};
    \node at (6,-1) {$X$};
    \node at (6,-4) {$X$};
    \node at (-2,-4) {$X$};
    
    \end{tikzpicture}     \caption{The projection of the Lagrangian surgery trace to the $\CC$ coordinate, as well as the ends of the Lagrangian surgery trace. The orange areas (measuring the surgery neck in the $X$ coordinate, and the shadow in the $\CC$ coordinate) agree.}
    \label{fig:lagrangiansurgerytrace}
\end{figure}

The shadow of a Lagrangian cobordism provides coarse bounds for the ``size'' of the set we can squeeze a Lagrangian cobordism into. 

For the three problems above, we will use Lagrangian cobordisms to provide constructions for squeezing problems, and remark where the shadow provides an obstruction to the smallest size sets we can squeeze our constructions into.

\begin{prop}\label{p:modifysuspension}
    Let $U\subset \CC$ be a connected open set with $\Area(U)>\Area(K)$. Then there exists $K'$ with $\Area(\pi_\CC(K')\setminus U)=0$ and $K'$ is compactly Hamiltonian isotopic to $K$.  
\end{prop}
\begin{proof}
    Let $V\subset \CC$ be an open set with the property that $\Area(\pi_\CC(K)\setminus V)=0$ and $\Area(V)<\Area(U)$. There exists a compactly supported Hamiltonian isotopy $\phi_t:\CC\to \CC$ so that $\phi_1(V)\subset U$. Extend this componentwise to a Hamiltonian isotopy $\tilde \phi_t:X\times \CC\to X\times \CC$; we define $K':=\tilde \phi_1(K)$.
\end{proof}
\begin{prop}
    Let $K_1: L\rightsquigarrow L_1$ and $K_2: L\rightsquigarrow L_2$ be two suspension cobordisms of time-dependent Hamiltonian isotopies with support in $(0, t_0)$. Let $c>0$, $\gamma_-\subset \CC$ be the U-shaped curve with $\gamma_-(0)=0, \gamma_-(1)=\sqrt{-1} \cdot c$ and $\gamma_+\subset \CC$ be the U-shaped curve with $\gamma_+(0)=t_0, \gamma_+ (1)=t_0+\sqrt{-1} \cdot c$ (as drawn in \cref{fig:circleLemma}).
    Suppose that $c,\gamma_-,\gamma_+$ are chosen such that the following two sets are submanifolds of $X \times \CC$ 
    \begin{align*}
        (L_1\times \gamma_-) \cup K_1|_{0<\Re(z)<t_0} \cup (K_1|_{0<\Re(z)<t_0}+\sqrt{-1} \cdot c)\cup (L\times \gamma_+)\\
        (L_2\times \gamma_-) \cup K_2|_{0<\Re(z)<t_0} \cup (K_2|_{0<\Re(z)<t_0}+\sqrt{-1} \cdot c)\cup (L\times \gamma_+)
    \end{align*}
    Then they are related by an exact isotopy.
    \label{lem:circlelemma}
\end{prop}
The argument is the same as the construction of $\overline i^L_s$ from \cite[Page 591]{hicks2021lagrangian}.
\begin{figure}
    \centering
    \begin{tikzpicture}

\fill[gray!20]  (-2,2) rectangle (4.5,-1.5);
\draw (-0.5,-0.5) .. node[fill=gray!20]{$\gamma^-$} controls (-1.5,-0.5) and (-1.5,1) .. (-0.5,1);

\begin{scope}[]

\draw[fill=orange!20] (-0.5,1) .. controls (0,1) and (0.5,1.5) .. (1,1.5) .. controls (1.5,1.5) and (2.5,1) .. (3,1) .. controls (2.5,1) and (1.5,0.5) .. (1,0.5) .. controls (0.5,0.5) and (0,1) .. (-0.5,1);

\end{scope}

\begin{scope}[shift={(0,-1.5)}]

\draw[fill=orange!20] (-0.5,1) .. controls (0,1) and (0.5,1.5) .. (1,1.5) .. controls (1.5,1.5) and (2.5,1) .. (3,1) .. controls (2.5,1) and (1.5,0.5) .. (1,0.5) .. controls (0.5,0.5) and (0,1) .. (-0.5,1);

\end{scope}

\begin{scope}[shift={(0,-1.5)}]
\draw (3,1) ..  node[fill=gray!20]{$\gamma^+$} controls (4.5,1)and (4.5,2.5) .. (3,2.5);

\end{scope}

\node at (1.25,1) {$K_{i}+\sqrt{-1}\cdot c$};
\node at (1.25,-0.5) {$K_{i}$};
\node at (4,1.5) {$\mathbb C$};
\end{tikzpicture}     \caption{Letting $K_i=K_1, K_2$ in the above diagram yields Lagrangian submanifolds related by an exact isotopy}
    \label{fig:circleLemma}
\end{figure}
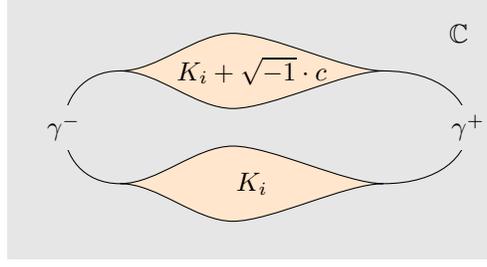
 \section{Construction One: Squeezing Lagrangian tori in  $P(a,b)$}\label{sec:packing}
In this section, we present Construction \ref{t:1}. The idea of the construction is to view $P(a,b)$ as a symplectic disk bundle over a disk. Then we concatenate several Lagrangian suspension cobordisms of circles in the disk fibers to produce a Lagrangian in $P(a,b)$ that helps us displace $L(1,1)$ from the integral packing.

\begin{proof}[Proof of Construction \ref{t:1}]
Let $X$ be the standard symplectic disk of area $a>2$ and $D$ be the standard symplectic disk of area $b>1$ so that $P(a,b)=X \times D$.
Let $\pi_X$ and $\pi_D$ be the projections to the first and second factors, respectively.
Let $C_{X,m}=\pi_X(L(m,n))$ and $C_{D,n}=\pi_D(L(m,n))$ so that $L(m,n)=C_{X,m} \times C_{D,n}$.
Let $C_{D,1}' \subset D$ be a small Hamiltonian perturbation of $C_{D,1}$ such that $C_{D,1} \pitchfork C_{D,1}'$ consists of two points.
Let $L(1,1)'=C_{X,1} \times C_{D,1}'$.
To show that the integral packing of $P(a,b)$ by $\{L(m,n)\}$ is not maximal, it suffices to find a Lagrangian $L$ in $P(a,b)$ that is Hamiltonian isotopic to $L(1,1)$, disjoint from $L(1,1)'$ and disjoint from $L(m,n)$ for all $(m,n) \in \mathbb{N}^2 \setminus \{(1,1)\}$.
We regard $D$ as a subset of $\CC$ and $P(a,b) \subset X \times \CC$.
A portion of $L$ will be the suspension of a Hamiltonian function of $X$ (see \cref{eq:suspension}).
\begin{figure}
    \centering
    \begin{subfigure}{.45\linewidth}
        \centering
        \scalebox{.7}{\begin{tikzpicture}[scale=2]

\draw[dashed, fill=gray!20] (0,0) circle[radius=1.65];
\draw[red] (0,0) circle[radius=1.41];
\draw[blue] (0,0) circle[radius=1.];
\draw[rounded corners, orange] (-20:.7)-- (-20:.5) arc (-20:20:.5)-- (20:1.1)  arc (20:160:1.1) --(160:1.35)arc (160:-160:1.35)--(-160:1.1)arc (-160:-20:1.1)-- (-20:.7)   ;

\node[ blue, right] at (-1,0) {$C_{X, 1}$};
\node[ red, left] at (-1.425,0) {$C_{X, 2}$};
\node[ orange, left] at (0.5,0) {$S_{X, \epsilon}$};

\clip (0,0) circle[radius=1.];
\fill[rounded corners, pattern=north west lines, pattern color=black] (-20:1.1)-- (-20:.5) arc (-20:20:.5)-- (20:1.1)  ;

\node[fill=gray!20] at (0.75,0) {$\epsilon$};
\end{tikzpicture} }
        \caption{Lagrangians in $X$.}
        \label{fig:LagX}
    \end{subfigure}
    \begin{subfigure}{.45\linewidth}
        \centering
        \scalebox{.7}{
        \begin{tikzpicture}[scale=2]

\draw[dashed, fill=gray!20] (0,0) circle[radius=1.65];
\draw[red] (0,0) circle[radius=1.41];
\draw[blue] (0,0) circle[radius=1.];
\node[ blue, right] at (-1,0) {$C_{X, 1}$};
\node[ red, left] at (-1.425,0) {$C_{X, 2}$};

\draw[rounded corners, orange] (-20:1.5)-- (-20:1.6) arc (-20:20:1.6)-- (20:1.36)  arc (20:160:1.36) --(160:1.1)arc (160:-160:1.1)--(-160:1.36)arc (-160:-20:1.36)-- (-20:1.46)   ;

\node[ orange, fill=gray!20] at (1.7,0) {$\phi(S_{X, \epsilon})$};
\end{tikzpicture} }
        \caption{A Hamiltonian isotopy of $S_{X, \epsilon}$}
        \label{fig:problem1isotopy}
    \end{subfigure}
    \begin{subfigure}{.45\linewidth}
        \centering
        \scalebox{.7}{
        \begin{tikzpicture}[scale=2]
\fill[gray!20] (0,0) circle[radius=1.3];

\fill[pattern=north west lines, pattern color=green] (0,0) circle[radius=1.];
\fill[gray!20]  (0.325,0) circle[radius=1];

\draw (0,0) circle[radius=1.];
\draw  (0.325,0) circle[radius=1];
\draw[orange]  (-0.25,0) ellipse (1 and 1);
\node[left] at (1,0) {$C_{D,1}$};
\node[right] at (-0.7,0) {$C_{D, 1}'$};
\node[orange, left] at (-1.225,0) {$S_{D}$};
\draw[dotted] (0,-1) ellipse (0.25 and 0.25);

\node[fill=gray!20, scale=.8] at (-0.825,-0.025) {$U_0$};
\end{tikzpicture} }
        \caption{Lagrangians in $D$}
        \label{fig:sdepsilon}
    \end{subfigure}
    \begin{subfigure}{.45\linewidth}
        \centering
        \scalebox{.7}{
        \begin{tikzpicture}[scale=10]

\draw[dotted] (0,-1) ellipse (0.25 and 0.25);
\clip (0,-1) ellipse (0.25 and 0.25);

\fill[gray!20] (0,0) circle[radius=1.3];

\fill[pattern=north west lines, pattern color=green] (0,0) circle[radius=1.];
\fill[gray!20]  (0.325,0) circle[radius=1];

\draw (0,0) circle[radius=1.];
\draw[orange]  (-0.25,0) ellipse (1 and 1);
\draw  (0.325,0) ellipse (1 and 1);
\node[fill=gray!20] at (1,0) {$C_{D,1}$};
\node[fill=gray!20] at (-0.7,0) {$C_{D, 1}'$};
\node[orange, fill=gray!20] at (-1.225,0) {$S_{D,\epsilon}$};
\fill[orange] (-0.107,-0.9895) .. controls (-0.0685,-0.984) and (-0.1625,-0.898) .. (-0.1385,-0.898) .. controls (-0.121,-0.898) and (-0.05,-0.98) .. (-0.003,-0.969) .. controls (-0.05,-0.98) and (-0.025,-0.9925) .. (-0.05,-0.9975) .. controls (-0.0705,-0.9925) and (-0.0685,-0.984) .. (-0.107,-0.9895);
\node at (-0.2,-1.025) {$S_{X, \epsilon}$};
\node at (-0.175,-0.875) {$C'_{D, 1}$};
\node at (0.2,-0.95) {$C_{D, 1}$};
\node at (0.125,-0.875) {$\phi(S_{X, \epsilon})$};
\node at (-0.125,-1.015) {$p$};
\node at (0.04,-0.935) {$p'$};
\end{tikzpicture} }
        \caption{A modification of $S_{D}$ at the inset area from \cref{fig:sdepsilon}. The labels correspond to the $X$-component of the Lagrangian.}
        \label{fig:problem1squeezing}
    \end{subfigure}
    \caption{}
\end{figure}
Let $U_0$ and $U_1$ be the two small (closed) bigons in $D$ with one side on $C_{D,1}$ and the other side on $C_{D,1}'$, and denote the area of $U_i$ by $\delta$. The $L$ we construct will have shadow inside $U_0$.

For $\epsilon>0$, let $S_{X,\epsilon} \subset X$ be an embedded circle such that
\begin{itemize}
    \item it is disjoint from $C_{X,m}$ for all $m \ge 2$
    \item it is Hamiltonian isotopic to $C_{X,1}$ by a compactly supported Hamiltonian
    \item the intersection between the disk bounded by $C_{X,1}$ and the disk bounded by $S_{X,\epsilon}$ is a simply connected region with area less than $\epsilon$
\end{itemize}
(see Figure \ref{fig:LagX}).
With the last assumption, if $a>2+\epsilon$, then we can find a compactly supported Hamiltonian $H_t \in C^{\infty}([0,1] \times X)$ such that $\|H_t\|=\epsilon$ and the time-$1$ flow
$\phi_{H_t}^1(S_{X,\epsilon}) \cap C_{X,1}=\emptyset$ (see \cref{fig:problem1isotopy}).
We choose $\epsilon<\min\{a-2, \delta/3\}$ and fix such a choice of $H$.

Let $S_D \subset D$  be a small Hamiltonian perturbation of $C_{D,1}$ such that 
\begin{itemize}
    \item it is disjoint from $C_{D,n}$ for all $n \ge 2$
    \item $S_D \cap C_{D,1} \cap C_{D,1}'=\emptyset$
    \item $S_D \pitchfork C_{D,1}$ consists of two points  $p,q \in \partial U_0$ 
    \item $S_D \pitchfork C_{D,1}'$ consists of two points $p',q' \in \partial U_0$  
\end{itemize}
By possibly relabeling, we assume that $S_D \cap U_0=I_{p,p'} \cup I_{q,q'}$ where $\partial I_{p,p'}=\{p,p'\}$
and $\partial I_{q,q'}=\{q,q'\}$.
We can also assume that $\overline{S_D \setminus (I_{p,p'} \cup I_{q,q'})}=I_{p,q} \cup I_{p',q'}$ where 
$\partial I_{p,q}=\{p,q\}$ and $\partial I_{p',q'}=\{p',q'\}$
(see Figure \ref{fig:sdepsilon}).

By construction, the product Lagrangian $L_S:=S_{X,\epsilon} \times S_D$ does not intersect $L(m,n)$ for all $(m,n) \neq (1,1)$.
It is also clear that $L_S$ is Hamiltonian isotopic to $L(1,1)$.
However, $L_S$ intersects $L(1,1)'$ non-trivially in the fibers $X \times \{p',q'\}$, so we need to modify $L_S$ to get our $L$.

To do that, recall that $3\epsilon < \delta=\operatorname{area}(U_0)$.
We replace $S_{X,\epsilon} \times I_{p,p'} \subset L_S$ by a suspension $L_{H,p}$ of $H_t$ such that (cf. Example \ref{e:suspensionarea} and Proposition \ref{p:modifysuspension})
\begin{itemize}
    \item $\pi_D(L_{H,p}) \subset U_0$ and it has area $\epsilon$,
    \item $L_{H,p} \cap X \times C_{D,1}=S_{X,\epsilon} \times \{p\}$
    \item $L_{H,p} \cap X \times C_{D,1}'=\phi_{H}(S_{X,\epsilon}) \times \{p'\}$
\end{itemize}
as drawn in \cref{fig:problem1squeezing}. By assumption, $\phi_{H}(S_{X,\epsilon}) \cap C_{X,1}=\emptyset$ so $L_{H,p} \cap L(1,1)' =\emptyset$.
Similarly,  we replace $S_{X,\epsilon} \times I_{q,q'} \subset L_S$ by a suspension $L_{H,q}$ of $H$ such that the corresponding conditions are satisfied.
Moreover, we can assume that $\pi_D(L_{H,p}) \cap \pi_D(L_{H,q}) =\emptyset$ because $3\epsilon < \delta$.

Our desired $L$ is the union of $S_{X,\epsilon} \times I_{p,q}$, $L_{H,p}$, $L_{H,q}$ and $\phi_{H}(S_{X,\epsilon}) \times I_{p',q'}$.
It is clear that $L$ is an embedded Lagrangian which is disjoint from $L(1,1)'$ and $L(m,n)$ for all $(m,n) \neq (1,1)$.
Moreover, there is an exact Lagrangian isotopy from $L$ to $L_S$ given by \cref{lem:circlelemma}.

\end{proof}

\begin{corollary}\label{c:packings}
When $a>2$ and $b>1$, there are $3$ pairwise disjoint integral Lagrangians in $P(a,b)$.
Moreover, two of them are Hamiltonian isotopic to $L(1,1)$ and the remaining one is Hamiltonian isotopic to $L(2,1)$.

When $a>2$ and $b>2$, there are $6$ pairwise disjoint integral Lagrangians in $P(a,b)$.
Moreover, two of them are Hamiltonian isotopic to $L(1,1)$, two of them are Hamiltonian isotopic to $L(1,2)$, one of them is Hamiltonian isotopic to $L(2,1)$ and the remaining one is Hamiltonian isotopic to $L(2,2)$.
\end{corollary}

\begin{proof}
The first statement is a direct consequence of the previous construction.
To explain the second statement, we use the notation in the previous proof.
Recall that the construction happens in a neighborhood of $X \times C_{D,1}$.
Therefore, when $b>2$ (i.e. the area of $D$ is greater than $2$), we can apply similar construction in a neighborhood of $X \times C_{D,2}$ to get an extra integral Lagrangian.
This implies the second statement.

\end{proof}
 \section{Construction Two: On displacing Lagrangian links in $S^2 \times S^2$}

In this section, we present Construction \ref{t:2}.
The nature of this problem is similar to the previous one. 
In fact, Construction \ref{t:2} can be viewed as a generalization of Construction \ref{t:1}.

\begin{proof}[Proof of Construction \ref{t:2}]

Let $X=S^2_{2B+C}$ and $D=S^2_{2a}$.
Assume $B-C \ge 0$; otherwise, the construction is obvious.
Let $K_a' \subset D$ be a circle that is Hamiltonian isotopic to $K_a$ such that $K_a \pitchfork K_a'$ consists of two points.
To complete the construction, it suffices to find a Lagrangian $L$ in $X \times D$ that is Hamiltonian isotopic to $L_1 \times K_a$
and disjoint from both $L_1 \times K_a'$ and $L_2 \times K_a$.

Let $U_0,U_1$ and $V_0,V_1$ be the four (closed) bigons in $D$ with one side on $K_a$ and the other side on $K_a'$ such that their interiors are pairwise disjoint.
By possibly relabeling, we assume that $U_0$ and $U_1$ have the same area, denoted by $\delta_U$,
and that $V_0$ and $V_1$ have the same area, denoted by $\delta_V$.
Therefore, we have $2a=2(\delta_U+\delta_V)$.
Without loss of generality, we assume that $\delta_U \ge \delta_V$.

For $\epsilon>0$, let $S_{X,\epsilon} \subset X$ be an embedded circle such that
\begin{itemize}
    \item it is disjoint from $L_2$
    \item it is Hamiltonian isotopic to $L_1$
    \item the intersection between the disk with area $B$ bounded by $L_1$ and the disk with area $B$ bounded by $S_{X,\epsilon}$ is a simply connected region with area less than $B-C+\epsilon$
\end{itemize}
With the last assumption, we can find a compactly supported Hamiltonian $H \in C^{\infty}([0,1] \times X)$ such that $\|H\|=B-C+\epsilon$ and 
$\phi_H(S_{X,\epsilon}) \cap L_1=\emptyset$.
We fix such a choice of $H$.

Let $S_D \subset D$ be a circle such that 
\begin{itemize}
    \item it is Hamiltonian isotopic to $K_a$
    \item $S_D \pitchfork K_a$ consists of two points  $p,q \in \partial U_0$ 
    \item $S_D \pitchfork K_a'$ consists of two points $p',q' \in \partial U_0$  
    \item it is disjoint from $U_1$
\end{itemize}
By possibly relabeling, we assume that $S_D \cap U_0=I_{p,p'} \cup I_{q,q'}$, $S_D \cap V_0=I_{p,q}$ and $S_D \cap V_1= I_{p',q'}$, where the subscripts of $I$ are the two endpoints.

By construction, the product Lagrangian $L_S:=S_{X,\epsilon} \times S_D$ does not intersect $L_2 \times K_a$ but it intersects $L_1 \times K_a'$ non-trivially in the fibers $X \times \{p',q'\}$, so we need to modify $L_S$ to get our $L$.

To do that, we apply the same suspension trick as in Construction \ref{t:1} to replace $S_{X,\epsilon} \times I_{p,p'}$, $S_{X,\epsilon} \times I_{q,q'}$ and $S_{X,\epsilon} \times I_{p',q'}$
by $L_{H,p}$, $L_{H,q}$ and $\phi_H(S_{X,\epsilon}) \times I_{p',q'}$, respectively.
This is possible when the area of the shadow, $2(B-C+\epsilon)$, is less than $\delta_U$. By letting $\delta_V$ be sufficiently small, $\delta_U$ can be arbitrarily close to $a$ so the construction works as long as 
$2(B-C)< a$.

\end{proof}

The proof of Construction \ref{t:2'} is entirely parallel.

\section{Construction Three: Lagrangian submanifolds in  $S^2\times S^2$ of low genus}
\subsection{Isotopic Lagrangians are Lagrangian cobordant}

Construction \ref{t:3} is based on a special instance of the following general observation.

\begin{lemma}
    Let $[\eta]\in H^1(L, \RR)$ be Poincar\'{e} dual to $[H]\in H_{1}(L, \ZZ)$, where $\eta$ is a closed $1$-form and $H$ is an embedded hypersurface of $L$. Let $L_\eta\in T^*L$ be the Lagrangian section given by the graph of $\eta$. There exists an embedded Lagrangian cobordism $K_{\eta,H}: L\rightsquigarrow L_\eta$.
    \label{lem:isotopicGivesCobordism}
\end{lemma}
\begin{proof}
    Let $H\times [-1, 1]$ be a normal neighborhood of $H$ in $L$. We can split $T^*(H\times [-1, 1])= T^*H\times T^*[-1, 1]$. Let $(q, p)$ be the coordinates on $T^*[-1, 1]$. 
    Consider the curves $\gamma_1, \gamma_2\subset T^*[-1, 1]$ drawn in \cref{fig:surgerysetup}, which are chosen so that the area contained by the curve is at least 30. The Lagrangians $H\times \gamma_1$ and $H\times \gamma_2$ are Hamiltonian isotopic. Let $K_1: H\times \gamma_1\rightsquigarrow H\times \gamma_2$ be the suspension cobordism.

     $(H\times \gamma_1)\cup (H\times \gamma_2)$ intersects $H\times [-1, 1]\subset L$ cleanly inside $T^*L$. Let $I_i:=[-1, 1]\cap \gamma_i$. Using \cite[Corollary 2.22]{mak2018dehn} we can perform a clean Lagrangian surgery, and let 
    \[L':=(H\times \gamma_1)\#_{H\times I_1} L \#_{H\times I_2}( H\times \gamma_2).\] We have a surgery trace cobordism
    \[K_2: (H\times \gamma_1, L, H\times \gamma_2)\rightsquigarrow L'\]
    The resulting surgery (drawn in \cref{fig:aftersurgery}) has two connected components, both of which are immersed. On the top connected component, perform an exact homotopy from $L'$ to $L''$ in the $T^*[-1, 1]$ factor  as indicated in \cref{fig:exacthomotopy}. The exact homotopy is constructed so that the area swept over the blue region is 1 and so that each of the red regions has an area of $1/2$. Let $K_{3}: L'\rightsquigarrow L''$ be the suspension of this exact homotopy. Let $K_4:L''\rightsquigarrow (H\times \gamma_1', L''', H\times \gamma_2')$ be the antisurgery trace cobordism from \cref{fig:exacthomotopy} to \cref{fig:finallagrangian}. The Lagrangians $H\times \gamma_1'$ and $H\times \gamma_2'$ are Hamiltonian isotopic; let $K_5$ be the suspension of this Hamiltonian isotopy. Note that 
    \[L'''=(L\setminus( H\times [-1, 1]))\cup (H\times \gamma_3)\]
    where $\gamma_3$ is indicated in \cref{fig:finallagrangian}.
    This is the graph of a section $\eta'$ of $T^*L$ with $[\eta']=PD(H)=[\eta]$. We can therefore find a Hamiltonian isotopy between $L'''$ and  $L_\eta$. Let $K_6: L'''\rightsquigarrow L_\eta$ be the suspension of this isotopy.

    Finally, let $K^{im}: L\rightsquigarrow L_\eta$ be the Lagrangian cobordism constructed by assembling $K_1, \ldots, K_6$ according to \cref{fig:assembly}. This is an immersed Lagrangian cobordism; by perturbing and resolving the double points we obtain an embedded Lagrangian cobordism $K_{\eta,H}: L\rightsquigarrow L_\eta$. 
\end{proof}
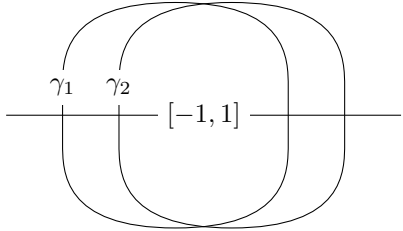
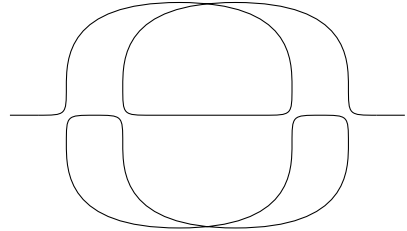
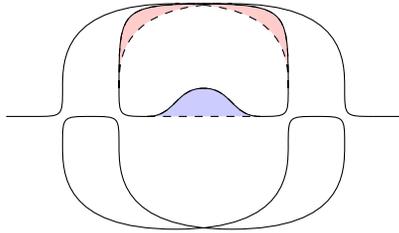
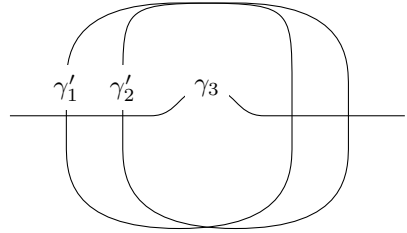
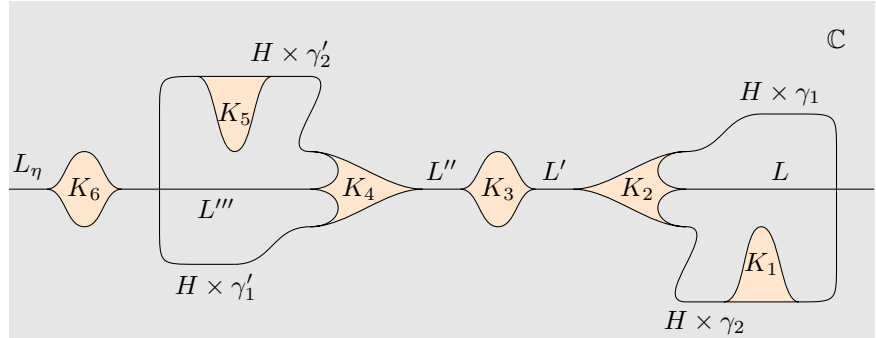
\begin{figure}\centering
    \begin{subfigure}{.4\linewidth}
    \centering
    \begin{tikzpicture}[scale=.75]

    \draw (-3.5,0) -- (0,0) node[fill=white] {$[-1, 1]$} -- (3.5,0);
    \begin{scope}[]
    \draw (-2.5,0.5) node[fill=white] {$\gamma_1$}.. controls (-2.5,1) and (-2.5,2) .. (-0.5,2) .. controls (1.5,2) and (1.5,1) .. (1.5,0.5) .. controls (1.5,0) and (1.5,0) .. (1.5,-0.5) .. controls (1.5,-1) and (1.5,-2) .. (-0.5,-2) .. controls (-2.5,-2) and (-2.5,-1) .. (-2.5,-0.5) .. controls (-2.5,0) and (-2.5,0) .. (-2.5,0.5);
    
    \end{scope}
    \begin{scope}[shift={(1,0)}]
    \draw (-2.5,0.5) node[fill=white] {$\gamma_2$}.. controls (-2.5,1) and (-2.5,2) .. (-0.5,2) .. controls (1.5,2) and (1.5,1) .. (1.5,0.5) .. controls (1.5,0) and (1.5,0) .. (1.5,-0.5) .. controls (1.5,-1) and (1.5,-2) .. (-0.5,-2) .. controls (-2.5,-2) and (-2.5,-1) .. (-2.5,-0.5) .. controls (-2.5,0) and (-2.5,0) .. (-2.5,0.5);
    
    \end{scope}
    
    \end{tikzpicture}     \caption{Projection of  $H\times\gamma_1, H\times \gamma_2, H\times [-1, 1]\subset T^*H\times T^*[-1, 1]$ to the $T^*[-1, 1]$ component. }
    \label{fig:surgerysetup}
    \end{subfigure}
    \hspace{.1\linewidth}
    \begin{subfigure}{.4\linewidth}
        \centering
        \begin{tikzpicture}[scale=.75]

\draw (-3,0) .. controls (-2.5,0) and (-2.5,0) .. (-2.5,0.5) .. controls (-2.5,1) and (-2.5,2) .. (-0.5,2) .. controls (1.5,2) and (1.5,1) .. (1.5,0.5) .. controls (1.5,0) and (1.5,0) .. (1,0) .. controls (0.5,0) and (-0.5,0) .. (-1,0) .. controls (-1.5,0) and (-1.5,0) .. (-1.5,0.5) .. controls (-1.5,1) and (-1.5,2) .. (0.5,2) .. controls (2.5,2) and (2.5,1) .. (2.5,0.5) .. controls (2.5,0) and (2.5,0) .. (3,0);
\draw (-2,0) .. controls (-2.5,0) and (-2.5,0) .. (-2.5,-0.5) .. controls (-2.5,-1) and (-2.5,-2) .. (-0.5,-2) .. controls (1.5,-2) and (1.5,-1) .. (1.5,-0.5) .. controls (1.5,0) and (1.5,0) .. (2,0) .. controls (2.5,0) and (2.5,0) .. (2.5,-0.5) .. controls (2.5,-1) and (2.5,-2) .. (0.5,-2) .. controls (-1.5,-2) and (-1.5,-1) .. (-1.5,-0.5) .. controls (-1.5,0) and (-1.5,0) .. (-2,0);
\draw (-3.5,0) -- (-3,0) (3,0) -- (3.5,0);
\end{tikzpicture}         \caption{Performing surgery at the clean self-intersections to obtain the Lagrangian $L'$.}
        \label{fig:aftersurgery}
    \end{subfigure}
    \begin{subfigure}{.4\linewidth}
        \centering
        \begin{tikzpicture}[scale=.75]

\draw[dashed, fill=red!20] (-1.5,0.5) .. controls (-1.5,1) and (-1.5,2) .. (0.5,2) .. controls (-1.5,2) and (-1.5,2) .. (-1.5,0.5);
\draw[dashed, fill=red!20] (1.5,0.5) .. controls (1.5,2) and (1.5,2) .. (-0.5,2) .. controls (1.5,2) and (1.5,1) .. (1.5,0.5);
\draw[dashed, fill=blue!20] (-1,0) .. controls (-0.5,0) and (0.5,0) .. (1,0) .. controls (0.5,0) and (0.5,0.5) .. (0,0.5) .. controls (-0.5,0.5) and (-0.5,0) .. (-1,0);

\draw (-3,0) .. controls (-2.5,0) and (-2.5,0) .. (-2.5,0.5) .. controls (-2.5,1) and (-2.5,2) .. (-0.5,2) .. controls (1.5,2) and (1.5,2) .. (1.5,0.5) .. controls (1.5,0) and (1.5,0) .. (1,0) ;
\draw (-1,0) .. controls (-1.5,0) and (-1.5,0) .. (-1.5,0.5) .. controls (-1.5,2) and (-1.5,2) .. (0.5,2) .. controls (2.5,2) and (2.5,1) .. (2.5,0.5) .. controls (2.5,0) and (2.5,0) .. (3,0);
\draw (-2,0) .. controls (-2.5,0) and (-2.5,0) .. (-2.5,-0.5) .. controls (-2.5,-1) and (-2.5,-2) .. (-0.5,-2) .. controls (1.5,-2) and (1.5,-1) .. (1.5,-0.5) .. controls (1.5,0) and (1.5,0) .. (2,0) .. controls (2.5,0) and (2.5,0) .. (2.5,-0.5) .. controls (2.5,-1) and (2.5,-2) .. (0.5,-2) .. controls (-1.5,-2) and (-1.5,-1) .. (-1.5,-0.5) .. controls (-1.5,0) and (-1.5,0) .. (-2,0);
\draw (-3.5,0) -- (-3,0) (3,0) -- (3.5,0);
\draw (-1,0) .. controls (-0.5,0) and (-0.5,0.5) .. (0,0.5) .. controls (0.5,0.5) and (0.5,0) .. (1,0);
\end{tikzpicture}         \caption{$L''$, which is exactly homotopic to $L'$. The red regions have an area of 1/2, and the blue region has an area of 1.}
        \label{fig:exacthomotopy}
    \end{subfigure} 
    \hspace{.1\linewidth}
    \begin{subfigure}{.4\linewidth}
        \centering
        \begin{tikzpicture}[scale=.75]
    \begin{scope}[]
    \draw (-2.5,0.5) node[fill=white] {$\gamma_1'$}.. controls (-2.5,1) and (-2.5,2) .. (-0.5,2) .. controls (1.5,2) and (1.5,2) .. (1.5,0.5) .. controls (1.5,0) and (1.5,0) .. (1.5,-0.5) .. controls (1.5,-1) and (1.5,-2) .. (-0.5,-2) .. controls (-2.5,-2) and (-2.5,-1) .. (-2.5,-0.5) .. controls (-2.5,0) and (-2.5,0) .. (-2.5,0.5);
    
    \end{scope}
    \begin{scope}[shift={(1,0)}]
    \draw (-2.5,0.5) node[fill=white] {$\gamma_2'$}.. controls (-2.5,2) and (-2.5,2) .. (-0.5,2) .. controls (1.5,2) and (1.5,1) .. (1.5,0.5) .. controls (1.5,0) and (1.5,0) .. (1.5,-0.5) .. controls (1.5,-1) and (1.5,-2) .. (-0.5,-2) .. controls (-2.5,-2) and (-2.5,-1) .. (-2.5,-0.5) .. controls (-2.5,0) and (-2.5,0) .. (-2.5,0.5);
    
    \end{scope}
    
\draw (-3.5,0) .. controls (-3,0) and (-1.5,0) .. (-1,0) .. controls (-0.5,0) and (-0.5,0.5) .. (0,0.5)node[fill=white] {$\gamma_3$} .. controls (0.5,0.5) and (0.5,0) .. (1,0) .. controls (1.5,0) and (3,0) .. (3.5,0);
\end{tikzpicture}         \caption{Performing antisurgery. The Lagrangians $\gamma_1'$ and $\gamma_2'$ are Hamiltonian isotopic.}
        \label{fig:finallagrangian}
    \end{subfigure}
    \begin{subfigure}{\linewidth}
        \centering
        \begin{tikzpicture}[xscale=-1,yscale=-1]

\fill[gray!20]  (2.5,-3) rectangle (14,1.5);
\node at (3,1) {$\mathbb C$};
\begin{scope}[shift={(0,0)}]
\draw (2.5,-1) --  node [above] {$L$}(5,-1);
\draw[fill=orange!20] (3.5,-2.5) .. controls (3.75,-2.5) and (4.2,-2.5) .. (4.5,-2.5) .. controls (4.2,-2.5) and (4.25,-1.5) .. (4,-1.5) .. controls (3.75,-1.5) and (3.75,-2.5) .. (3.5,-2.5);
\node at (4,-2) {$K_1$};
\draw[fill=orange!20] (5,-1.5) .. controls (5.5,-1.5) and (5.5,-1) .. (5,-1) .. controls (5.5,-1) and (5.5,-0.5) .. (5,-0.5) .. controls (5.5,-0.5) and (6,-1) .. (6.5,-1) .. controls (6,-1) and (5.5,-1.5) .. (5,-1.5);
\draw (3.5,-2.5) .. controls (3,-2.5) and (3,-2.5) .. (3,-2) .. controls (3,-1) and (3,-1) .. (3,-0.5) .. controls (3,0) and (3,0) .. (3.5,0);
\draw (5,-2.5) .. controls (5.5,-2.5) and (4.5,-1.5) .. (5,-1.5);
\draw (4,0) .. controls (4.5,0) and (4.5,-0.5) .. (5,-0.5);
\draw (3.5,0) -- node [below] {$H\times \gamma_1$} (4,0);
\draw (4.5,-2.5) -- node [above] {$H\times \gamma_2$} (5,-2.5);

\end{scope}
\begin{scope}[]

\draw[fill=orange!20] (7,-1) .. controls (7.25,-1) and (7.25,-0.5) .. (7.5,-0.5) .. controls (7.75,-0.5) and (7.75,-1) .. (8,-1) .. controls (7.75,-1) and (7.75,-1.5) .. (7.5,-1.5) .. controls (7.25,-1.5) and (7.25,-1) .. (7,-1);

\node at (7.5,-1) {$K_3$};
\end{scope}
\begin{scope}[rotate=180, shift={(-15,2)}]
\draw (2.5,-1) --  node [below] {$L'''$}(5,-1);
\draw[fill=orange!20] (3.5,-2.5) .. controls (3.75,-2.5) and (4.25,-2.5) .. (4.5,-2.5) .. controls (4.25,-2.5) and (4.25,-1.5) .. (4,-1.5) .. controls (3.75,-1.5) and (3.75,-2.5) .. (3.5,-2.5);
\node at (4,-2) {$K_5$};
\draw[fill=orange!20] (5,-1.5) .. controls (5.5,-1.5) and (5.5,-1) .. (5,-1) .. controls (5.5,-1) and (5.5,-0.5) .. (5,-0.5) .. controls (5.5,-0.5) and (6,-1) .. (6.5,-1) .. controls (6,-1) and (5.5,-1.5) .. (5,-1.5);
\draw (3.5,-2.5) .. controls (3,-2.5) and (3,-2.5) .. (3,-2) .. controls (3,-1) and (3,-1) .. (3,-0.5) .. controls (3,0) and (3,0) .. (3.5,0);
\draw (5,-2.5) .. controls (5.5,-2.5) and (4.5,-1.5) .. (5,-1.5);
\draw (4,0) .. controls (4.5,0) and (4.5,-0.5) .. (5,-0.5);
\draw (3.5,0) -- node [below] {$H\times \gamma_1'$} (4,0);
\draw (4.5,-2.5) -- node [below] {$H\times \gamma_2'$} (5,-2.5);

\end{scope}

\draw (6.5,-1) --  node[above]{$L'$}(7,-1);
\draw (8,-1) -- node[above]{$L''$} (8.5,-1);
\node at (5.65,-1) {$K_2$};
\node at (9.35,-1) {$K_4$};

\begin{scope}[shift={(5.5,0)}]

\draw[fill=orange!20] (7,-1) .. controls (7.25,-1) and (7.25,-0.5) .. (7.5,-0.5) .. controls (7.75,-0.5) and (7.75,-1) .. (8,-1) .. controls (7.75,-1) and (7.75,-1.5) .. (7.5,-1.5) .. controls (7.25,-1.5) and (7.25,-1) .. (7,-1);

\node at (7.5,-1) {$K_6$};
\end{scope}
\draw (13.5,-1) -- node[above]{$L_\eta$} (14,-1);
\end{tikzpicture}         \caption{Using the cobordisms $K_1, \ldots, K_6$ to obtain a cobordism from $L$ to $L_\eta$. The labels on the diagram give the Lagrangian in $X$ above the corresponding portion of the Lagrangian cobordism}
        \label{fig:assembly}
    \end{subfigure}
    \caption{Construction of a Lagrangian cobordism associated with a Lagrangian isotopy with integral flux class.}
    \label{fig:isotopygivescobordism}
\end{figure}
In a discussion with \`{A}lvaro Mu\~niz Brea, it was observed that the construction of \cref{lem:isotopicGivesCobordism} gives a non-orientable Lagrangian cobordism. 
\begin{corollary}\label{c:LagIsotopic}
Let $X$ be a compact symplectic manifold. Whenever $L_0, L_1\subset X$ are orientable and Lagrangian isotopic, there exists a Lagrangian cobordism $K: L_1\rightsquigarrow L_0$.
\end{corollary}
\begin{proof}
First, suppose that $L$ is the zero section of $T^*L$, and that $L_\eta$ is some other section given by the graph of a 1-form $\eta$. Write $[\eta]=\sum_{j=1}^k [\eta_j]$ where each $[\eta_j]=\alpha_j PD([H_j])$ for some embedded hypersurface $H_j\subset L$.
Let $L_i$ be the graph of $\sum_{j=1}^i \eta_j$. Observe that (after applying a symplectomorphism), $L_{i+1}$ is a Lagrangian section of $T^*L_i$ satisfying the criterion of \cref{lem:isotopicGivesCobordism} (where we modify the construction by scaling the area of the blue region in \cref{fig:exacthomotopy} to be $\alpha_{i+1}$). Thus, we can concatenate the Lagrangian cobordisms from \cref{lem:isotopicGivesCobordism} to construct a Lagrangian cobordism $K_\eta$ from $L$ to $L_\eta$.

For $L_0, L_1\subset X$, take a subdivision $0=t_0<\cdots < t_k=1$, so that:
\begin{itemize}
    \item There exists a Weinstein neighborhood $U$ of $L_{t_i}$ for which $L_{t_{i+1}}$ is the graph of a section $L_{\eta_i}\subset U\subset T^*L_{t_i}$.
    \item The $\eta_i$ are small enough so that the Lagrangian cobordisms $K_{\eta_i}:L_{t_i}\rightsquigarrow L_{\eta_i}$ constructed above fit entirely within $U\times \CC\subset T^*L_{t_i}\times \CC$.
\end{itemize}
Because $X$ is compact, such a subdivision can be found. Then by taking $K$ to be the concatenation of the $K_{\eta_i}$ we obtain a Lagrangian cobordism from $L_1$ to $L_0$.
\end{proof}

We say that a Lagrangian submanifold $K: L_1\rightsquigarrow L_0$ satisfying the properties of \cref{lem:isotopicGivesCobordism} is ``topologically efficient'' if
\[\underline b_1(K):=\dim(H_1(K,L_1))\]
 is close to $0$.  This can also be expressed as $\dim(H_1^{Morse}(K, -\pi_\RR|_{K}))$ (where the differential counts negative gradient flowlines and we use $\ZZ_2$ coefficients). Observe that when $K$ is topologically $L_0\times \RR$ then $\underline b_1(K)=0$.
 
 The construction given in \cref{lem:isotopicGivesCobordism} is not very topologically efficient. A more efficient (although more lengthy to explicitly describe) cobordism from $L\rightsquigarrow L'''$ can be built by generalizing Lagrangian antisurgery disks from \cite{haug2020lagrangian}. Let $\overline \gamma$ be the upper semicircle belonging to the curve $\gamma_1$ from \cref{fig:surgerysetup}, so that $H\times \overline \gamma$ is a Lagrangian with boundary cleanly meeting $H\times [-1, 1]$. This Lagrangian gives anti-surgery data generalizing the data of a Lagrangian antisurgery disk considered by \cite{haug2020lagrangian}. By performing antisurgery along $H\times \overline \gamma$, we obtain a Lagrangian $\overline L'$, which is the upper connected component of $L'$ from \cref{fig:aftersurgery}. We then perform the same exact isotopy as in \cref{fig:exacthomotopy}, and then a clean surgery to obtain $L'''$.

A different measure of the ``efficiency'' of a Lagrangian cobordism is the shadow of the cobordism. Yet another measure of distance between \emph{isotopic} Lagrangians can be computed using the flux of a symplectic isotopy, which measures the symplectic area swept by 1-cycles of $L$ over the isotopy.

There is a trade-off between topologically efficient and shadow efficient Lagrangian cobordisms. 
In the simplest example that we consider (where $\eta= \alpha PD([H])$), the constructed cobordism $K_{\eta,H}$ has shadow $\Area(K_{\eta, H})>\alpha$. Let us additionally assume that the Lagrangian cobordism $K_6$ in the construction of $K_{\eta, H}$ is trivial, so we can assume that $\Area(K_{\eta, H})\approx A\cdot\alpha$ for some fixed constant of proportionality $A$.

An easy way to lower the shadow of the cobordism we construct is to take $n$ disjoint copies of $H$, so that $\eta=\frac{\alpha}{n} PD[H]$. The construction we give has a shadow of approximately 
\[\Area(K_{\eta, nH})\approx \frac{A\cdot \alpha}{n}\]
If additionally we assume that $H$ is connected, this is proportional to $\frac{\alpha}{\underline b_1(K_{\eta, nH})}$. So, the shadow can be traded for the topological complexity of a cobordism. 

More generally, an argument due to Emmy Murphy \cite[Section 4.4]{cornea2019lagrangian} shows that given $K: L_1\rightsquigarrow L_0$, for all $C>0$ there exists a Lagrangian cobordism $K': L_1\rightsquigarrow L_0$ with $\Area(K')<C$ (but $\underline b_1(K')\gg \underline b_1(K)$). 

Given $L_0, L_1$ which are Lagrangian homotopic, let $\mathcal I(L_0, L_1):=\{i_t: L\times I\to X\;|\; i_0(L)=L_0, i_1(L)=L_1, \omega|_{L_t}=0\}$ denote the set of Lagrangian homotopies between $L_0$ and $L_1$. Let $\mathcal B$ denote the set of bases for $H_1(L, \ZZ)$. We look at the minimal flux
    \begin{align*}
        d(L_0, L_1):= \inf_{\substack{i_t\in \mathcal I(L_0, L_1)\\ B\in \mathcal B}} \sum_{c\in B}| \Flux_{i_t}(c)|.
    \end{align*}
\begin{conjecture}
    There exists a constant $A$ so that for any pair of Lagrangian isotopic submanifolds  $L_0, L_1$  and Lagrangian cobordism $K: L_1\rightsquigarrow L_0$, we have 
    \begin{equation}
        \Area(K)> \frac{A\cdot d(L_0, L_1)}{ \underline b_1(K)} .\label{eq:conj}
    \end{equation}
\end{conjecture}
A reverse inequality cannot hold uniformly for all choices of cobordism $K$ with fixed ends and $\underline b_1(K)>0$: increasing the topological complexity of a cobordism can make the right-hand side of \cref{eq:conj} arbitrarily small without changing $d(L_0,L_1)$. However, our construction was motivated by the following specific example exhibiting the reverse bound. 
\begin{example}
    \label{exam:s1isotopy}
    For this example, we let $X=T^*S^1$. 
Let $S^1_c\subset X$ be the Lagrangian submanifold given by the graph of $c\cdot d\theta$,
so that $d(S^1_{\lambda/2},S^1_{-\lambda/2})=\lambda$ as in \cref{fig:s1lambda}. Pick $\epsilon>0$ and $n\in \NN$. We now describe a Lagrangian null cobordism $\tilde K_{\lambda,n,\epsilon}: (S^1_{\lambda/2}\cup S^1_{-\lambda/2})\rightsquigarrow \emptyset$, with the property that 
\[\frac{\lambda}{2n}<\Area(\tilde K_{\lambda, n, \epsilon})<\frac{\lambda}{2n}+\epsilon\]
and $\underline b_1(\tilde K_{\lambda, n, \epsilon})=4n+1$.

For fixed $n\in \NN$, consider the Lagrangian $L_1$ which is the zig-zag curve drawn in \cref{fig:twosquiggles}. $L_1$ crosses the zero section $2n$ times. Let $L_2$ be the reflection of $L_1$ across the zero section. The Lagrangians $L_i$ are constructed so that the regions indicated in \cref{fig:twosquiggles} satisfy $2A+B+C=\lambda/(2n)$. The value of $C$ can be taken as small as desired.
The Lagrangian $L_1$ is Hamiltonian isotopic to the zero section, and the Hofer norm of the Hamiltonian isotopy is $A$. Let $K_1: L_1\to S^1_0$ be the suspension of this Hamiltonian isotopy. Similarly, let $K_2: L_2\to S^1_0$ be the suspension of the Hamiltonian isotopy from $L_2$ to the zero section. Then $\Area(K_1)=\Area(K_2)=A$.
\begin{figure}
    \centering
    \begin{subfigure}{.26\linewidth}
        \centering
        \begin{tikzpicture}

    \fill[gray!20]  (3,3) rectangle (6,-1);
    
\begin{scope}[]

\clip  (6,-0.42) rectangle (3,2.43);
\fill[orange!20] (3.5,1) -- (3.6,2.5) -- (4,1);
\fill[purple!20] (3.6,-0.5) -- (4,1) -- (4.4,-0.5);
\fill[yellow!20] (4.6,2.5) -- (4.5,1) -- (4.4,2.5);
\end{scope}
    
\begin{scope}[]
\draw[red, rounded corners] (3,1) -- (3.4,-0.5) -- (3.6,2.5)-- (4,1);
\draw[blue, rounded corners] (3,1) -- (3.4,2.5) -- (3.6,-0.5)  -- (4,1);
   
\end{scope}
\begin{scope}[shift={(2,0)}]
\draw[red, rounded corners] (3,1) -- (3.4,-0.5) -- (3.6,2.5)-- (4,1);
\draw[blue, rounded corners] (3,1) -- (3.4,2.5) -- (3.6,-0.5)  -- (4,1);
   
\end{scope}
\begin{scope}[shift={(1,0)}]
\draw[red, rounded corners] (3,1) -- (3.4,-0.5) -- (3.6,2.5)-- (4,1);
\draw[blue, rounded corners] (3,1) -- (3.4,2.5) -- (3.6,-0.5)  -- (4,1);
   
\end{scope}
\node at (3.64,1.4) {$A$};
\node at (4,0.1) {$B$};
\node at (4.5,2.56) {$C$};
\node[] at (3,1) {$\times$};
\node[] at (4,1) {$\times$};
\node[] at (5,1) {$\times$};

\node[circle, fill=black, scale=.2] at (3.5,1) {};
\node[circle, fill=black, scale=.2] at (4.5,1) {};
\node[circle, fill=black, scale=.2] at (5.5,1) {};

\draw[|-|] (2.5,2.4) -- node[fill=white]{$\lambda$} (2.5,-0.4);
    \draw (3,3) -- (3,-1);
    \draw (6,3) -- (6,-1);
    \draw[dashed] (3,3) -- (6,3) (3,-1) -- (6,-1);
    
    \node[blue] at (3.4,2.8) {$L_2$};
    \node[red] at (3.4,-0.7) {$L_1$};
\end{tikzpicture}         \caption{}
        \label{fig:twosquiggles}
    \end{subfigure}
    \begin{subfigure}{.22\linewidth}
        \centering
        \begin{tikzpicture}

    \fill[gray!20]  (3,3) rectangle (6,-1);
    
\begin{scope}[]

\clip  (6,-0.42) rectangle (3,2.43);
\fill[purple!20] (3.6,-0.5) -- (4,1) -- (4.4,-0.5);
\fill[yellow!20] (4.6,2.5) -- (4.5,1) -- (4.4,2.5);
\end{scope}
\begin{scope}
\clip  (6.1,-0.4) rectangle (2.7,-0.7);
\draw[fill=green!20,rounded corners] (3.5,1) -- (3.6,-0.5) -- (4,-0.5) -- (4.4,-0.5) -- (4.5,1);

\end{scope}
    
\node at (4,0.1) {$B$};
\node at (4.5,2.56) {$C$};

\node[ fill=black, scale=.2] at (3.5,1) {};
\node[ fill=black, scale=.2] at (4.5,1) {};
\node[ fill=black, scale=.2] at (5.5,1) {};

    \draw (3,3) -- (3,-1);
    \draw (6,3) -- (6,-1);
    \draw[dashed] (3,3) -- (6,3) (3,-1) -- (6,-1);

\draw[rounded corners] (3,-0.5) -- (3.4,-0.5) -- (3.6,2.5) -- (4,2.5) -- (4.4,2.5) -- (4.6,-0.5) -- (5,-0.5) -- (5.4,-0.5) -- (5.6,2.5) -- (6,2.5);
\draw[rounded corners] (3,2.5) -- (3.4,2.5) -- (3.6,-0.5) -- (4,-0.5) -- (4.4,-0.5) -- (4.6,2.5) -- (5,2.5) -- (5.4,2.5) -- (5.6,-0.5) -- (6,-0.5);

\end{tikzpicture}         \caption{}
        \label{fig:firstSurgery}
    \end{subfigure}
    \begin{subfigure}{.22\linewidth}
        \centering
        \begin{tikzpicture}

    \fill[gray!20]  (3,3) rectangle (6,-1);
    
\begin{scope}[]

\clip  (6,-0.42) rectangle (3,2.43);
\fill[purple!20] (3.6,-0.4) -- (4,1) -- (4.4,-0.4);
\fill[yellow!20] (4.6,2.4) -- (4.5,1) -- (4.4,2.4);
\end{scope}
\begin{scope}
\clip  (6.1,-0.4) rectangle (2.7,-0.7);
\draw[fill=green!20,rounded corners] (3.5,1) -- (3.6,-0.4) -- (4,-0.4) -- (4.4,-0.4) -- (4.5,1);

\end{scope}
    
\node at (4,0.1) {$B$};
\node at (4.5,2.46) {$C$};

\node[ fill=black, scale=.2] at (3.5,1) {};
\node[ fill=black, scale=.2] at (4.5,1) {};
\node[ fill=black, scale=.2] at (5.5,1) {};

    \draw (3,3) -- (3,-1);
    \draw (6,3) -- (6,-1);
    \draw[dashed] (3,3) -- (6,3) (3,-1) -- (6,-1);

\draw[rounded corners] (3,-0.4) -- (3.4,-0.4) -- (3.6,2.4) -- (4,2.4) -- (4.4,2.4) -- (4.6,-0.4) -- (5,-0.4) -- (5.4,-0.4) -- (5.6,2.4) -- (6,2.4);
\draw[rounded corners] (3,2.4) -- (3.4,2.4) -- (3.6,-0.4) -- (4,-0.4) -- (4.4,-0.4) -- (4.6,2.4) -- (5,2.4) -- (5.4,2.4) -- (5.6,-0.4) -- (6,-0.4);

\end{tikzpicture}         \caption{}
        \label{fig:isotopy}
    \end{subfigure}
    \begin{subfigure}{.22\linewidth}
        \centering
        \begin{tikzpicture}

    \fill[gray!20]  (3,3) rectangle (6,-1);

    \draw (3,3) -- (3,-1);
    \draw (6,3) -- (6,-1);
    \draw[dashed] (3,3) -- (6,3) (3,-1) -- (6,-1);

    \draw (3,-0.4) --  node[fill=gray!20] {$S^1_{-\lambda/2}$}(6,-0.4);
\draw (3,2.5) -- node[fill=gray!20] {$S^1_{\lambda/2}$}(6,2.5);
\end{tikzpicture}         \caption{}
        \label{fig:s1lambda}
    \end{subfigure}
    \caption{
        Some Lagrangian submanifolds in $X=T^*S^1$. (a) Two zig-zag Lagrangians when $n=3$. The intersection points $x_i$ are marked with $\times$, and the $y_i$ are marked with circles. (b) The slice $K_b|_0$ obtained from surgering the two zig-zags at the $x_i$. The green region has area $C+\epsilon/8$. (c) A Lagrangian obtained after applying an exact homotopy which removes the green region. (d) The slice $K_d|_{-t_0}$, which is obtained from the previous Lagrangian by surgering at the $y_i$.
        }
\end{figure}
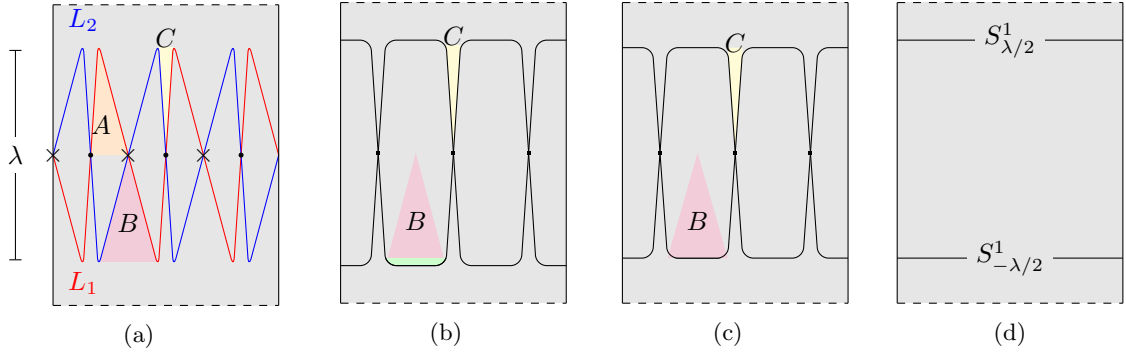
The Lagrangians $L_1$ and $L_2$ intersect transversely at $2n$ points. Label the intersection points $x_i=(i/n, 0)$ and $y_i=(i/n+1/2n, 0)$ for $i\in\{ 0, \ldots, n-1\}$.
Consider the Lagrangian submanifold $L_1 {}_{x_i}\#_{y_i} L_2$ which is obtained from $L_1$ and $L_2$ by performing surgery
\begin{itemize}
    \item with $L_2$ as the first factor and $L_1$ as the second factor for surgeries performed at $\{x_i\}$ with surgery size $B$,
    \item with $L_1$ as the first factor and $L_2$ as the second factor at surgeries performed at $\{y_i\}$ with surgery size $C$.
\end{itemize}
For these choices of Lagrangian surgeries, $L_1 {}_{x_i}\#_{y_i} L_2= (S^1_{\lambda/2}\cup S^1_{-\lambda/2})$.

We now describe an embedded Lagrangian cobordism $K_3$ whose ends are $(L_1,L_2)\rightsquigarrow L_1 {}_{x_i}\#_{y_i} L_2$. The construction is similar to the description of the surgery trace (\cref{exam:surgery}), but requires an additional argument due to the differing order of branches employed in the Polterovich surgery.
First, consider the Lagrangian $K_a:=(L_1\times \ell_1)\cup (L_2\times \ell_2)$ as drawn in \cref{fig:detail1}. All of the self-intersections of $K_a$ lie in the $t=0$ slice, where $K_a|_{t=0}=L_1\cup L_2$. 
\begin{enumerate}[series=surgeries, label=(s\arabic*)]
    \item Resolve the self-intersections of the form $\{(x_i,0)\}$ using Polterovich surgery with surgery neck size $B+C+\epsilon/8$ to obtain $K_b$, drawn in \cref{fig:detail2}. \label{s:1}
\end{enumerate}
The self-intersections of $K_b$ still live in the slice $K_b|_0$, and are in bijection with $\{y_i\}\times (\ell_1\cap \ell_2)$. We draw the slice (\cref{claim:slice}) $K_b|_0$ in \cref{fig:firstSurgery}.

Pick a value $-t_0$ so that the area bounded by the lines $\Re(z)=-t_0$, $\Re(z)=0$ and the shadow of $K_b$ is $C+\epsilon/8$ (each green region from \cref{fig:detail2}).
Using the discussion following \cite[Example 4.2.2]{hicks2021lagrangian}, reparameterize the Lagrangian cobordism $K_b$ between the slices $-t_0$ and $0$ to create a bottleneck at $-t_0$. 
Call the resulting Lagrangian $K_c$ (\cref{fig:detail3}). This Lagrangian is immersed with transverse intersections, and corresponds to the suspension of an exact homotopy between the slices $-t_0$ and $0$.  We have that $K_b|_0=K_c|_0$, and the slice $K_c|_{-t_0}$ is drawn in \cref{fig:isotopy}. Observe that \cref{fig:firstSurgery,fig:isotopy} differ by the light green regions (also of area $C+\epsilon/8$).
We now apply Lagrangian surgery to the self-intersections of $K_c$ at $t=0$ and $t=-t_0$.
\begin{enumerate}[resume=surgeries, label=(s\arabic*)]
    \item At the self-intersections in the slice $t=0$, perform surgery at all the points with a surgery neck width of $\epsilon/8$. This corresponds to the small red region in \cref{fig:detail4}. \label{s:2}
    \item At the self-intersections in the slice $t=-t_0$, perform surgery with neck width of $C$. This corresponds to the small yellow region drawn in \cref{fig:detail4}.\label{s:3}
\end{enumerate}
Call the subsequent Lagrangian submanifold $K_d$. The dark gray region belonging to $K_d$ has area $(C+\epsilon/8)-C=\epsilon/8$. The slice $K_d|_{-t_0}=L_1 {}_{x_i}\#_{y_i} L_2$ as desired. The Lagrangian submanifold $K_3$ is obtained by truncating $K_d$ at time $-t_0$ in the sense of \cite[Definition 3.1.7]{hicks2021lagrangian}. We have from inspection that 
\begin{align*}
    \Area(K_3)=&\textcolor{yellow}{C}+\textcolor{gray}{\epsilon/8}+\textcolor{red}{2\epsilon/8}+\textcolor{green}{(C+\epsilon/8)}+\textcolor{purple}{B}\\
    =&B+2C+\epsilon/2
\end{align*}
\begin{figure}
    \centering
    \begin{subfigure}{.23 \linewidth}
        \centering
        \begin{tikzpicture}
\fill[gray!20]  (-3,-1.5) rectangle (0.5,0.5);
    \begin{scope}[shift={(0,0)}]
        \node at (-2.5,0) {$L_1\times \ell_1$};
        \node at (-2.5,-1) {$L_2\times \ell_2$};
        \draw (-2.25,0.5) -- (-0.25,-1.5) (-0.25,0.5) -- (-2.25,-1.5);
        
        \end{scope}
        
\draw[dotted] (-1.25,0.5) -- (-1.25,-1.5);
\node[below] at (-1.25,-1.5) {$0$};
\node at (0,0) {$\mathbb C$};
\end{tikzpicture}         \caption{$K_a$}
        \label{fig:detail1}
    \end{subfigure}
    \begin{subfigure}{.23 \linewidth}
        \centering

\begin{tikzpicture}
\fill[gray!20]  (-3,-1.5) rectangle (0.5,0.5);
    \begin{scope}[shift={(-3.75,0)}]
        \fill[purple!20] (1.5,0.5) .. controls (2,0) and (2.5,-0.5) .. (2.5,-0.5) .. controls (2.5,-0.5) and (2,-1) .. (1.5,-1.5) .. controls (2,-1) and (2,0) .. (1.5,0.5);    
	\fill[green!20] (2,0) -- (2,-1) -- (2.5,-0.5);
	\draw (1.5,0.5) .. controls (2,0) and (2.5,-0.5) .. (2.5,-0.5) .. controls (2.5,-0.5) and (2,-1) .. (1.5,-1.5) .. controls (2,-1) and (2,0) .. (1.5,0.5);    

        \fill[purple!20] (3.5,0.5) .. controls (3,0) and (2.5,-0.5) .. (2.5,-0.5) .. controls (2.5,-0.5) and (3,-1) .. (3.5,-1.5) .. controls (3,-1) and (3,0) .. (3.5,0.5);
        \fill[green!20] (3,0) -- (3,-1) -- (2.5,-0.5);
        \draw (3.5,0.5) .. controls (3,0) and (2.5,-0.5) .. (2.5,-0.5) .. controls (2.5,-0.5) and (3,-1) .. (3.5,-1.5) .. controls (3,-1) and (3,0) .. (3.5,0.5);

        \end{scope}

\node at (0,0) {$\mathbb C$};

\draw[dotted] (-1.75,0.5) -- (-1.75,-1.5);
\node[below] at (-1.75,-1.5) {$-t_0$};
\end{tikzpicture}         \caption{$K_b$}
        \label{fig:detail2}
    \end{subfigure}
    \begin{subfigure}{.23 \linewidth}
        \centering

\begin{tikzpicture}
\fill[gray!20]  (-3,-1.5) rectangle (0.5,0.5);

\node at (0,0) {$\mathbb C$};

\draw[dotted] (-1.75,0.5) -- (-1.75,-1.5);
\node[below] at (-1.75,-1.5) {$-t_0$};
    \begin{scope}[shift={(-3.75,0)}]

        \fill[purple!20] (3.5,0.5) .. controls (3,0) and (2.5,-0.5) .. (2.5,-0.5) .. controls (2.5,-0.5) and (3,-1) .. (3.5,-1.5) .. controls (3,-1) and (3,0) .. (3.5,0.5);
        \fill[green!20] (3,0) -- (3,-1) -- (2.5,-0.5);
        \draw (3.5,0.5) .. controls (3,0) and (2.5,-0.5) .. (2.5,-0.5) .. controls (2.5,-0.5) and (3,-1) .. (3.5,-1.5) .. controls (3,-1) and (3,0) .. (3.5,0.5);

        \draw[fill=green!20] (2.5,-0.5) .. controls (2.4,-0.6) and (2.4,-1.2) .. (2.3,-1.2) .. controls (2.2,-1.2) and (2.1,-0.6) .. (2,-0.5) .. controls (2.1,-0.4) and (2.2,0.1) .. (2.3,0.1) .. controls (2.4,0.1) and (2.4,-0.4) .. (2.5,-0.5);
        \end{scope}
\begin{scope}[]
\draw[fill=purple!20] (-1.8,-0.5) .. controls (-1.87,-1) and (-1.9,-1.4) .. (-2,-1.4) .. controls (-2.1,-1.4) and (-2.15,-1.4) .. (-2.25,-1.5) .. controls (-1.75,-1) and (-1.75,0) .. (-2.25,0.5) .. controls (-2.15,0.4) and (-2.1,0.3) .. (-2,0.3) .. controls (-1.9,0.3) and (-1.86,0) .. (-1.8,-0.5);

\draw[fill=purple!20] (-1.8,-0.5) .. controls (-1.78,-1) and (-1.78,-1) .. (-1.75,-0.5);
\draw[fill=purple!20] (-1.8,-0.5) .. controls (-1.78,0) and (-1.78,0) .. (-1.75,-0.5);

\end{scope}

\end{tikzpicture}         \caption{$K_c$}
        \label{fig:detail3}
    \end{subfigure}
    \begin{subfigure}{.23 \linewidth}
        \centering

\begin{tikzpicture}
\fill[gray!20]  (-3,-1.5) rectangle (0.5,0.5);

\draw[fill=red!20] (-1.32,-0.3) .. controls (-1.3,-0.4) and (-1.2,-0.4) .. (-1.05,-0.3) .. controls (-1.2,-0.4) and (-1.2,-0.6) .. (-1.05,-0.7) .. controls (-1.2,-0.6) and (-1.3,-0.6) .. (-1.32,-0.7);

\node at (0,0) {$\mathbb C$};

\begin{scope}[]
\draw[fill=purple!20] (-1.8,-0.5) .. controls (-1.87,-1) and (-1.9,-1.4) .. (-2,-1.4) .. controls (-2.1,-1.4) and (-2.15,-1.4) .. (-2.25,-1.5) .. controls (-1.75,-1) and (-1.75,0) .. (-2.25,0.5) .. controls (-2.15,0.4) and (-2.1,0.3) .. (-2,0.3) .. controls (-1.9,0.3) and (-1.86,0) .. (-1.8,-0.5);

\draw[fill=purple!20] (-1.8,-0.5) .. controls (-1.78,-1) and (-1.78,-1) .. (-1.75,-0.5);
\draw[fill=purple!20] (-1.8,-0.5) .. controls (-1.78,0) and (-1.78,0) .. (-1.75,-0.5);

\end{scope}
\draw[dotted] (-1.75,0.5) -- (-1.75,-1.5);
\node[below] at (-1.75,-1.5) {$-t_0$};
    \begin{scope}[shift={(-3.75,0)}]

        \fill[purple!20] (3.5,0.5) .. controls (3,0) and (2.5,-0.5) .. (2.5,-0.5) .. controls (2.5,-0.5) and (3,-1) .. (3.5,-1.5) .. controls (3,-1) and (3,0) .. (3.5,0.5);
        \fill[green!20] (3,0) -- (3,-1) -- (2.5,-0.5);
        \draw (3.5,0.5) .. controls (3,0) and (2.5,-0.5) .. (2.5,-0.5) .. controls (2.5,-0.5) and (3,-1) .. (3.5,-1.5) .. controls (3,-1) and (3,0) .. (3.5,0.5);

        \draw[fill=gray!50] (2.5,-0.5) .. controls (2.4,-0.6) and (2.4,-1.2) .. (2.3,-1.2) .. controls (2.2,-1.2) and (2.1,-0.6) .. (2,-0.5) .. controls (2.1,-0.4) and (2.2,0.1) .. (2.3,0.1) .. controls (2.4,0.1) and (2.4,-0.4) .. (2.5,-0.5);
      \draw[fill=yellow!20] (2,-0.5) .. controls (2.1,-0.4) and (2.2,0.1) .. (2.3,0.1) .. controls (2.2,-0.4) and (2.2,-0.7) .. (2.3,-1.2) .. controls (2.2,-1.2) and (2.1,-0.6) .. (2,-0.5);
\end{scope}

\end{tikzpicture}         \caption{$K_d$}
        \label{fig:detail4}
    \end{subfigure}
    \caption{Some shadows of Lagrangian cobordisms. (a) Two trivial suspensions. (b) Surgering the Lagrangians at their intersection. The green regions have area $C+\epsilon/8$, while the purple region has area $B$. (c) Flipping the left end to create a bottleneck. Areas are preserved. (d) Applying a surgery at the intersections in the slices $-t_0$ and $0$. The yellow region has area $C$.}
\end{figure}
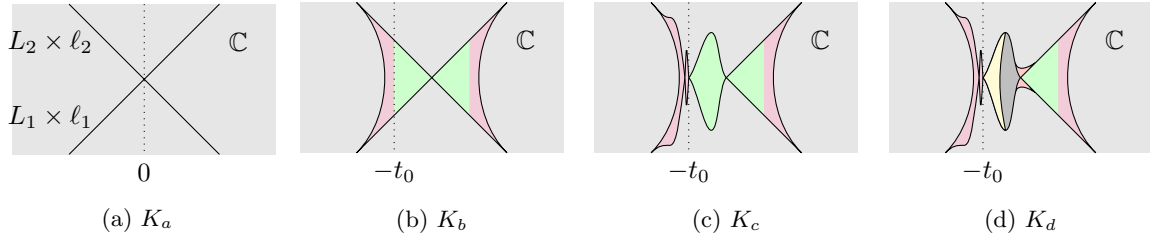

Finally, let $K_4: (S^1_0, S^1_0)\rightsquigarrow \emptyset$ be a U-shaped null cobordism.
The concatenation $\tilde K_{\lambda, n, \epsilon}:=K_4\circ (K_1\cup K_2)\circ (K_3)^{-1}: (S^1_{\lambda/2}\cup S^1_{-\lambda/2})\rightsquigarrow \emptyset$ drawn in \cref{fig:concatenation2} is a null cobordism. The concatenation is performed so the ``gap'' in the middle of the cobordism has area $\epsilon/2$. In this way, 
\begin{align*}
   \frac{\lambda}{2n}=& 2A+B+C<\Area(\tilde K_{\lambda, n, \epsilon})\\
   =&
   \Area(K_1)+\Area(K_2)+\Area(K_3)+\Area(\text{Gap in \cref{fig:concatenation2}})\\
   =&A+A+(B+2C+\epsilon/2)+\epsilon/2= \frac{\lambda}{2n}+\epsilon+C \stepcounter{equation}\tag{\theequation}
    \label{eq:shadowinequality}
\end{align*}
We recall that $\epsilon, C$ were independently chosen as small as desired. For convenience, we merge these two constants into one, so that $\Area(\tilde K_{\lambda, n, \epsilon})<\lambda/(2n)+\epsilon$.
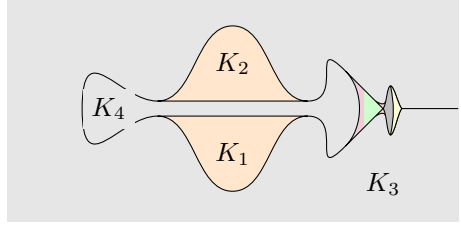
\begin{figure}
    \centering
    \begin{tikzpicture}[scale=1.2,xscale=-1]

    \fill[gray!20]  (12,1) rectangle (18,-2);
    \begin{scope}[shift={(11.75,-0.25)}, scale=0.5]    
        
\draw[fill=red!20] (2.43,-0.3) .. controls (2.45,-0.4) and (2.55,-0.4) .. (2.7,-0.3) .. controls (2.55,-0.4) and (2.55,-0.6) .. (2.7,-0.7) .. controls (2.55,-0.6) and (2.45,-0.6) .. (2.43,-0.7);

        \fill[purple!20] (3.5,0.5) .. controls (3,0) and (2.5,-0.5) .. (2.5,-0.5) .. controls (2.5,-0.5) and (3,-1) .. (3.5,-1.5) .. controls (3,-1) and (3,0) .. (3.5,0.5);
        \fill[green!20] (3,0) -- (3,-1) -- (2.5,-0.5);
        \draw (3.5,0.5) .. controls (3,0) and (2.5,-0.5) .. (2.5,-0.5) .. controls (2.5,-0.5) and (3,-1) .. (3.5,-1.5) .. controls (3,-1) and (3,0) .. (3.5,0.5);

        \draw[fill=gray!50] (2.5,-0.5) .. controls (2.4,-0.6) and (2.4,-1.2) .. (2.3,-1.2) .. controls (2.2,-1.2) and (2.1,-0.6) .. (2,-0.5) .. controls (2.1,-0.4) and (2.2,0.1) .. (2.3,0.1) .. controls (2.4,0.1) and (2.4,-0.4) .. (2.5,-0.5);
      \draw[fill=yellow!20] (2,-0.5) .. controls (2.1,-0.4) and (2.2,0.1) .. (2.3,0.1) .. controls (2.2,-0.4) and (2.2,-0.7) .. (2.3,-1.2) .. controls (2.2,-1.2) and (2.1,-0.6) .. (2,-0.5);
\end{scope}

    \begin{scope}[shift={(-1.5,0.9)}]
    
    \draw[fill=orange!20] (15.5,-1.5) .. controls (16,-1.5) and (16,-1.5) .. (16.5,-1.5) .. controls (17,-1.5) and (17,-1.5) .. (17.5,-1.5) .. controls (17,-1.5) and (17,-2.5) .. (16.5,-2.5) .. controls (16,-2.5) and (16,-1.5) .. (15.5,-1.5);
    
    \node at (16.5,-2) {$K_2$};
    \end{scope}
    
    \begin{scope}[shift={(-1.5,1.1)}]
    
    \draw[fill=orange!20] (15.5,-1.5) .. controls (16,-1.5) and (16,-0.5) .. (16.5,-0.5) .. controls (17,-0.5) and (17,-1.5) .. (17.5,-1.5) .. controls (17,-1.5) and (17,-1.5) .. (16.5,-1.5) .. controls (16,-1.5) and (16,-1.5) .. (15.5,-1.5);
    
    \node at (16.5,-1) {$K_1$};
    \end{scope}

    \draw (13.5,0) .. controls (14,0.5) and (13.5,-0.4) .. (14,-0.4);
    \draw (13.5,-1) .. controls (14,-1.5) and (13.5,-0.6) .. (14,-0.6);

    \draw (16,-0.6) .. controls (16.5,-0.6) and (17,-1.5) ..  (17,-0.5)node[fill=gray!20, right] {$K_4$} .. controls (17,0.5) and (16.5,-0.4) .. (16,-0.4);

    \node at (13,-1.5) {$K_3$};
\draw (12,-0.5) -- (12.75,-0.5);
\end{tikzpicture}
     \caption{Concatenation to form a null-cobordism of $S^1_{-\lambda/2}\cup S^1_{\lambda/2}$. The center region has area $\epsilon/2$.}
    \label{fig:concatenation2}
\end{figure}
By rearranging the ends of $\tilde K_{\lambda, n, \epsilon}$, we obtain a Lagrangian cobordism $K_{\lambda,n,\epsilon}: S^1_{\lambda/2}\rightsquigarrow S^1_{-\lambda/2}$. 

We now compute $\underline b_1$ using the ordinary Euler characteristic of the non-compact cobordisms.
$K_3$ is obtained from $(L_1\times \ell_1)\cup (L_2\times \ell_2)$, which has Euler characteristic zero, by resolving $2n$ transverse double points corresponding to the points $x_i$ and $y_i$.
Resolving one transverse double point of a surface removes two disks and glues in an annulus, and therefore decreases the Euler characteristic by $2$.
Hence
\[
    \chi(K_3)=0-2(2n)=-4n.
\]
The additional pieces used to form $\tilde K_{\lambda,n,\epsilon}$ from $K_3$ have Euler characteristic zero and are glued along circles, so they do not change the Euler characteristic.
Rearranging the ends also leaves the underlying surface unchanged.
Therefore
\[
    \chi(K_3)
    =
    \chi(\tilde K_{\lambda,n,\epsilon})
    =
    \chi(K_{\lambda,n,\epsilon})
    =
    -4n.
\]
For the null cobordism $\tilde K_{\lambda,n,\epsilon}$, the pair defining $\underline b_1$ is taken relative to the two positive ends.
Since these are all of the ends, the compactification has $H_2(\tilde K_{\lambda,n,\epsilon},S^1_{\lambda/2}\cup S^1_{-\lambda/2};\ZZ_2)\cong \ZZ_2$, so the same Euler-characteristic computation gives $\underline b_1(\tilde K_{\lambda,n,\epsilon})=4n+1$.
Since $K_{\lambda,n,\epsilon}$ has one positive end and one negative end, the long exact sequence of the pair $(K_{\lambda,n,\epsilon},S^1_{\lambda/2})$, with $\mathbb Z_2$-coefficients gives
\[
    \underline b_1(K_{\lambda,n,\epsilon})
    =
    \dim H_1(K_{\lambda,n,\epsilon},S^1_{\lambda/2})
    =
    -\chi(K_{\lambda,n,\epsilon})
    =
    4n.
\]

From this construction, we obtain Lagrangians satisfying 
\[2\frac{d(S^1_{\lambda/2},S^1_{-\lambda/2})}{\underline b_1(K_{\lambda, n, \epsilon})}=\frac{\lambda}{2n}<\Area(K_{\lambda, n,\epsilon})<\frac{\lambda}{2n}+\epsilon= 2\frac{d(S^1_{\lambda/2},S^1_{-\lambda/2})}{\underline b_1(K_{\lambda, n,\epsilon})}+\epsilon.\]
\end{example}
\subsection{Application:}

We now use the previous discussion to explain Construction \ref{t:3}.

\begin{proof}

Let $n$ be the smallest integer greater than $\lambda/2$, and choose $\epsilon$ small enough so that 
\[\frac{\lambda-\epsilon}{2n}+2\epsilon<1.\]
Consider the decomposition
\[S^2_1\times S^2_\lambda=(D^2_\epsilon\times S^2_{\lambda})\cup(D^2_{1-\epsilon} \times S^2_\lambda),\] where $D^2_A$ is the symplectic disk with area $A$. In the first component, use \cite{mikhalkin2019examples} to construct $K_1$, the Lagrangian lift of the tropical curve drawn on the left-hand side of \cref{fig:twohalves}. The topology of $K_1$ is a two-punctured sphere whose boundary is  $(\{\bullet\}\times (S^1_{(-\lambda+\epsilon)/2}\cup S^1_{(\lambda-\epsilon)/2}))\subset \{\bullet\}\times S^2_\lambda$.
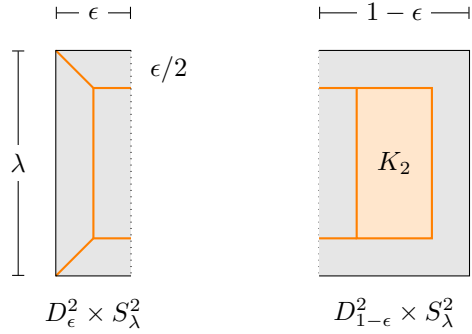
\begin{figure}
\centering
\begin{tikzpicture}

\fill[gray!20]  (-4,1.5) rectangle (-3,-1.5);
\draw[orange, thick] (-4,1.5) -- (-3.5,1) -- (-3,1);
\draw[orange, thick] (-4,-1.5) -- (-3.5,-1) -- (-3,-1);
\draw[orange, thick] (-3.5,-1) -- (-3.5,1);
\draw (-3,1.5) -- (-4,1.5) -- (-4,-1.5) -- (-3,-1.5);
\draw[dotted] (-3,1.5) -- (-3,-1.5);
\node at (-3.5,-2) {$D^2_\epsilon\times S^2_\lambda$};
\draw[dotted] (-0.5,1.5) -- (-0.5,-1.5);
\fill[gray!20]  (-0.5,1.5) rectangle (1.5,-1.5);
\draw (-0.5,1.5) -- (1.5,1.5) -- (1.5,-1.5) -- (-0.5,-1.5);
\draw[orange, thick] (-0.5,-1) -- (0,-1);
\draw[orange, thick] (-0.5,1) -- (0,1);
\draw[fill=orange!20, draw=orange, thick]  (1,-1) rectangle (0,1);
\node at (0.5,0) {$K_{2}$};
\node at (0.5,-2) {$D^2_{1-\epsilon}\times S^2_\lambda$};
\draw[|-|] (-4.5,1.5) -- node[fill=white]{$\lambda$} (-4.5,-1.5);

\draw[|-|] (-4,2) -- node[fill=white]{$\epsilon$} (-3,2);

\draw[|-|] (-2.5,1.5) -- node[fill=white]{$\epsilon/2$} (-2.5,1);

\draw[|-|] (-0.5,2) -- node[fill=white]{$1-\epsilon$} (1.5,2);
\end{tikzpicture} \caption{Dividing the $S^2_1\times S^2_\lambda$ into two portions. The left-hand side $K_1$ is a 2-punctured sphere. The right-hand side $K_2$ is described in \cref{exam:s1isotopy}.}
\label{fig:twohalves}
\end{figure}
Let $B^*_{\lambda}S^1$ be the portion of the cotangent bundle of $S^1$ between $S^1_{-\lambda/2}$ and $S^1_{\lambda/2}$.
On the right-hand side of \cref{fig:twohalves}, consider the Lagrangian null cobordism:
\[\tilde K_{\lambda-\epsilon, n,\epsilon}: S^1_{(-\lambda+\epsilon)/2}\cup S^1_{(\lambda-\epsilon)/2} \rightsquigarrow \emptyset\] from \cref{exam:s1isotopy}. Observe that
\[\Area(\tilde K_{\lambda-\epsilon, n,\epsilon})<\frac{\lambda-\epsilon}{2n}+\epsilon<1-\epsilon.\]
Therefore, after interchanging the factors in the local chart, this Lagrangian cobordism fits inside $D^2_{1-\epsilon}\times B^*_\lambda S^1\subset D^2_{1-\epsilon}\times S^2_\lambda$. We call this Lagrangian $K_2$.
$K_1$ and $K_2$ share the same boundary, so we can define $K=K_1\cup K_2$. 
This lies in the $\ZZ/2\ZZ$ homology class of $[\{\bullet\}\times S^2_\lambda]$ and has a non-orientable genus $4n+2$.

\end{proof}
 
\printbibliography
\end{document}